\documentclass[11pt]{article}

\usepackage[numbers]{natbib}
\usepackage{amsthm,amsmath,mathrsfs}
\usepackage{amssymb,amsthm,amsmath}
\usepackage[dvips]{geometry}
\usepackage{amscd}
\usepackage{amsfonts}
\usepackage[latin1]{inputenc}
\usepackage[all] {xy}
\usepackage[normalem]{ulem}
\usepackage{hyperref}
\usepackage{mathtools}
\usepackage[dvips]{geometry}
\usepackage{enumitem}
\usepackage{graphicx}
\usepackage{subcaption}
\usepackage{tikz-cd}
\usepackage[nottoc]{tocbibind}
\usepackage[titletoc,toc,title]{appendix}
\newtheorem{theorem}{Theorem}[section]
\newtheorem{maintheorem}{Theorem}

\newtheorem{proposition}[theorem]{Proposition}
\newtheorem{corollary}[theorem]{Corollary}

\newtheorem{lemma}[theorem]{Lemma}

\newtheorem{definition}[theorem]{Definition}

\newtheorem*{conjecture}{Conjecture}
\newtheorem{conjecturen}{Conjecture}
\newtheorem{remark}[theorem]{Remark}
\newtheorem{q}{Question}

\newcommand{\T}{\mathbb{T}}

\newcommand{\Z}{\mathbb{Z}}
\newcommand{\N}{\mathbb{N}}
\newcommand{\R}{\mathbb{R}}
\newcommand{\C}{\mathscr{C}}

\newcommand{\address}{{
  \bigskip
  \footnotesize

  \textsc{Department of Mathematics, CNRS UMR 8088, Universit\'e de Cergy-Pontoise, 2, av.  Adolphe Chauvin F-95302 Cergy-Pontoise, France } \par\nopagebreak
}}

\title{Uniqueness of the measure of maximal entropy for the standard map}
\author{Davi Obata\footnote{D.O. was supported by the projects ANR BEKAM : ANR-15-CE40-0001 and ERC project 692925 NUHGD.}
}
\date{}
\begin{document}

\maketitle
\begin{abstract}
In this paper we prove that for sufficiently large parameters the standard map has a unique measure of maximal entropy (m.m.e.). Moreover, we prove: the m.m.e. is Bernoulli, and the periodic points with Lyapunov exponents bounded away from zero equidistribute with respect to the m.m.e. We prove some estimates regarding the Hausdorff dimension of the m.m.e. and about the density of the support of the measure on the manifold. For a generic large parameter, we prove that the support of the m.m.e. has Hausdorff dimension $2$. We also obtain the $C^2$-robustness of several of these properties.
\end{abstract}
\setcounter{tocdepth}{1}
\tableofcontents

\section{Introduction}

A general goal in dynamical systems is to understand the asymptotic statistical and topological behavior of the orbits of a system. In general, this is a hard problem even for systems having a simple expression.

An example of a dynamical system with simple expression but having complex behavior is given by the famous standard map (or Taylor-Chirikov standard map). Considering $\T^2 = \R^2/ \Z^2$, with coordinates $(x,y) \in [0,1)^2$, for each $k\in \R$, the standard map is defined as 
\[\begin{array}{rcc}
f_k: \T^2 & \longrightarrow & \T^2\\
(x,y) & \mapsto & (2x - y + k \sin(2\pi x), x).
\end{array}
\] 
This diffeomorphism was introduced independently by Taylor and Chirikov (see \cite{chirikovstandard}), and it is related to many physical problems (see for instance \cite{chirikovstandard, istandard, sstandard}). For every $k$ the diffeomorphism $f_k$ preserves the usual Lebesgue measure on $\T^2$. A famous conjecture made by Sinai is the following:

\begin{conjecturen}[\cite{sinaibook}, page $144$]
\label{conjecture.sinai}
For $k$ large enough, the standard map $f_k$ has positive metric entropy for the Lebesgue measure. 
\end{conjecturen}

This conjecture remains open. Indeed, it is not known if there exists one parameter $k$ for which the standard map has positive metric entropy. By Pesin's entropy formula \cite{ch1pesin77}, this is equivalent on the existence of a set, of positive volume, of points having a positive Lyapunov exponent\footnote{See Section \ref{sec.prel} for the definition of Lyapunov exponents.}.

 Recently, Berger and Turaev proved (in \cite{bergerturaev}) that the standard map can be $C^{\infty}$-approximated by volume preserving diffeomorphisms having positive metric entropy. Another type of result in the theory is the positivity of the Lyapunov exponent for certain types of ``random'' perturbations of the standard map. In \cite{bluxueyoung}, Blumenthal, Xue and Young proved the positivity of the Lyapunov exponent for some random perturbations of the standard map; and in \cite{bergercarrasco} Berger and Carrasco proved the non-uniform hyperbolicity of a partially hyperbolic skew product derived from the standard map.

In general, the dynamics of the standard map itself is not well understood. Let us mention some of the known results. Since for $k=0$ the standard map is completely integrable, due to KAM theory, for small parameters many invariant circles persist. For such parameters there are some results for the standard map regarding separation of separatrices \cite{gelfreich99, gellaz}, invariant Cantor sets \cite{mmp}, and others (see \cite{delallave} for more references).

Duarte proved in \cite{duarte} that for a generic large parameter, there exists a ``topologically large'' uniformly hyperbolic set for the standard map which is accumulated by elliptic islands\footnote{See Section \ref{sec.standardback} for more details.}. Duarte's result was improved by Gorodetski in \cite{gorodetski}, who proved, for a generic large parameter, the existence of an increasing sequence of uniformly hyperbolic sets whose Hausdorff dimension converges to $2$ and which is accumulated by elliptic islands\footnote{See Theorem \ref{thm.gorod} for a precise statement.}. In 2013, Duarte's result was also further improved by De Simoi in \cite{desimoi}, where he proved that on one hand the set of sufficiently large parameters having infinitely many elliptic islands of a certain type (called cyclicity-one) has zero Lebesgue measure; on the other hand this same set of parameters contains a residual set and has positive Hausdorff dimension. See also \cite{matheusmoreirapalis, bloorluzzatto} for some other results for the standard map. In particular, these results illustrates some of the difficulties from studying the standard map, since elliptic islands will be accumulating on hyperbolic sets.

The goal of this work is to contribute for the understanding of dynamical properties of the standard map for large parameters. In particular, we prove the uniqueness of the measure of maximal entropy, and several other properties for this measure.

\subsection*{Measures of maximal entropy and main results}

For a diffeomorphism $f:M\to M$, a probability measure $\mu$ is \textbf{invariant} if for any measurable set $B$ we have $\mu(B) = \mu(f^{-1}(B))$. An invariant probability measure is \textbf{ergodic} if any measurable $f$-invariant set has $\mu$-measure $0$ or $1$. Let $\mathbb{P}_e(f)$ be the set of ergodic invariant probability measures for $f$. The well known variational principle (see for instance \cite{mane}) states that
\[
h_{\mathrm{top}}(f) = \sup_{\mu \in \mathbb{P}_e(f)} h_{\mu}(f),
\] 
where $h_{\mathrm{top}}(f)$ is the \textbf{topological entropy} of $f$, and $h_{\mu}(f)$ is the \textbf{metric entropy} of $f$ with respect to $\mu$\footnote{See \cite{mane} for the precise definition of topological and metric entropy.}.

\begin{definition}
An invariant measure $\mu$ for $f$ is a \textbf{measure of maximal entropy} (or \textbf{m.m.e.}) if $h_{\mathrm{top}}(f) = h_{\mu}(f)$. 
\end{definition}

Ergodic measures of maximal entropy are important in the theory. They may give, for instance, information about the asymptotic growth and equidistribution of periodic points (see for instance \cite{burguet}). Since almost every ergodic component of a measure of maximal entropy is also a measure of maximal entropy, it is a natural problem to understand about the finiteness of ergodic measures of maximal entropy. The study of existence and finiteness of ergodic measures of maximal entropy has a long history, which we will not try to state here. We refer the reader to Section $1.9$ in \cite{buzzicrovisiersarig} for many references on the history of this problem.

Newhouse proved in \cite{newhouse} that a $C^{\infty}$ system always have at least one ergodic m.m.e. In particular, for any $k\in \R$, the standard map $f_k$ has at least one m.m.e. In a remarkable recent work, Buzzi, Crovisier and Sarig obtained that a $C^{\infty}$-diffeomorphism of a compact surface with positive topological entropy has at most finitely many ergodic m.m.e. Moreover, if the diffeomorphism is transitive\footnote{Recall that $f$ is transitive if there exists a point with dense orbit.} then there is only one m.m.e. (see \cite{buzzicrovisiersarig}). Since for a generic large parameter the standard map has elliptic islands (see \cite{duarte}), it is not transitive for these parameters. In particular, Buzzi-Crovisier-Sarig's result does not give the uniqueness of the m.m.e. for the standard map. 

A natural question is to know if for sufficiently large parameters, the standard map has a unique m.m.e. Recall that an $f$-invariant measure $\mu$ is \textbf{Bernoulli} for $f$, if $(f,\mu)$ is isomorphic to a Bernoulli shift. Our main result is the following.

\begin{maintheorem}
\label{thm.maintheoremuniqueness}
There exists $k_0\in \R$ such that for $k\in [k_0, +\infty)$, the standard map $f_k$ has a unique m.m.e. which is Bernoulli. Moreover, this property is $C^2$-robust, that is, if $g\in \mathrm{Diff}^2(\T^2)$ is sufficiently $C^2$-close to $f_k$, then $g$ has at most one m.m.e., and if such a measure exists for $g$, then it is also Bernoulli.
\end{maintheorem}

We can actually obtain several other information about the m.m.e. in a neighborhood of the standard map. 

For a given diffeomorphism $f$, for each $\rho>0$ and $n\in \N$, let $\mathrm{Per}^{\rho}_n(f)$ be the set of periodic points of period $n$ whose Lyapunov exponents are bounded away from zero by at least $\rho$ (see Section \ref{subsec.pesintheory} for precise definitions). A point in $\mathrm{Per}^{\rho}_n(f)$ is called a \textbf{$\rho$-hyperbolic periodic point}. We obtain the following result.

\begin{maintheorem}
\label{thm.maintheoremperiodicpoints}
Let $k$ and $\mathcal{U}$ be as in Theorem \ref{thm.maintheoremuniqueness}, and let $g\in \mathcal{U}\cap \mathrm{Diff}^{\infty}(\T^2)$. Then $g$ has a unique m.m.e., $\mu_g$ , that verifies the following: for any $\rho\in (0, h_{\mathrm{top}}(g))$,
\begin{enumerate}
\item \textbf{Growth rate of the $\rho$-hyperbolic periodic points}:
\[
\displaystyle \lim_{n \to +\infty} \frac{1}{n} \log \# \mathrm{Per}^{\rho}_n(g) = h_{\mathrm{top}}(g).
\]
\item \textbf{Equidistribution of $\rho$-hyperbolic periodic orbits with respect to $\mu_g$}: 
\[
\displaystyle \lim_{n\to +\infty} \frac{1}{\# \mathrm{Per}^{\rho}_n(g)} \left( \sum_{p\in \mathrm{Per}^{\rho}_n(g)} \delta_{p} \right) = \mu_g,
\]
where $\delta_p$ is the dirac mass on $p$, and the limit above is for the $\mathrm{weak}^*$-topology.
\end{enumerate}
\end{maintheorem}

We remark that the reason why in Theorem \ref{thm.maintheoremperiodicpoints} we can only take $g\in \mathrm{Diff}^{\infty}(\T^2)$ is because in the proof we use a recent result by Burguet in \cite{burguet}, which requires $C^{\infty}$-regularity.

Using some estimates given by Duarte in \cite{duarte}, we may also obtain more information about how ``large'' the m.m.e. is on the manifold. In what follows, given a subset $X$ of the manifold $\T^2$, we write $\mathrm{dim}_H(X)$ the Hausdorff dimension of $X$. The Hausdorff dimension of a probability measure $\mu$ is defined as
\begin{equation}
\label{eq.hausmeasure}
\mathrm{dim}_H(\mu) := \inf \{dim_H(X): \mu(X) =1\}.
\end{equation}
We obtain the following.

\begin{maintheorem}
\label{thm.maintheoremdensity}
For every $\varepsilon >0$, there exists $k_0\in \R$ such that for any $k\in [k_0,+\infty)$ the following holds: there is a $C^2$-neighborhood $\mathcal{U}$ of $f_k$ such that for any $g\in \mathcal{U}$, if $\mu_g$ is a m.m.e. for $g$ then:
\begin{enumerate}
\item \textbf{Dimension of $\mu_g$:} $\mathrm{dim}_H(\mu_g) > 2-\varepsilon$.
\item \textbf{Density of $\mathrm{supp}(\mu_g)$:} the support of $\mu_g$ is $\frac{8}{k^{\frac{1}{3}}}$-dense in $\T^2$.  
\end{enumerate}
\end{maintheorem}

Suppose that $(g_n)_{n\in \N}$ is a sequence of $C^{\infty}$-diffeomorphisms of a surface converging in the $C^{\infty}$-topology to $g$. Let $\mu_n$ be a m.m.e. for $g_n$, for each $n\in \N$. Then any accumulation for the $\mathrm{weak}^*$-topology of the sequence $(\mu_n)_{n\in \N}$ is a m.m.e. for $g$ (see for instance Section $5.1$ in \cite{buzzicrovisiersarig}). Let $\mathbb{P}(\T^2)$ be the set of probability measures of $\T^2$ endowed with the $\mathrm{weak}^*$-topology. An immediate consequence of the fact above and our Theorem \ref{thm.maintheoremuniqueness} is that for large $k$, the unique m.m.e. for $f_k$ varies continuously with $k$.

\begin{corollary}
\label{cor.continuitymme}
For each $k$ large enough, let $\mu_{\mathrm{max}}(k)$ be the unique m.m.e. for $f_k$. There exists $k_0\in \R$ such that the map
\[
\begin{array}{ccl}
[k_0, +\infty) & \longrightarrow & \mathbb{P}(\T^2)\\
k & \mapsto & \mu_{\mathrm{max}}(k)
\end{array}
\]
is continuous.
\end{corollary}

Combining our methods with Gorodetski's main result in \cite{gorodetski} (see Theorem \ref{thm.gorod} in Section \ref{sec.prel}), we obtain the following result.

\begin{maintheorem}
\label{thm.genericparamenters}
There exist $k_0$ and a dense $G_{\delta}$-subset of $[k_0, +\infty)$, $\mathcal{R}$, such that for any $k\in \mathcal{R}$, the Hausdorff dimension of the support of the unique m.m.e. for $f_k$ is $2$.
\end{maintheorem}

From our proof and the results from \cite{buzzicrovisiersarig}, we actually obtain that for a large enough parameter, the standard map has a unique transitive invariant compact set that ``contains'' every measure with high enough entropy (see Remark \ref{remark.homomeasures} below).

\subsection*{Comments and strategy of the proof of Theorem \ref{thm.maintheoremuniqueness}}

In \cite{bergercarrasco}, Berger and Carrasco introduced a volume preserving partially hyperbolic skew product on $\T^4$ which is derived from the standard map on the fibers. The skew product gives an additional ``transversality'' which allows the authors to prove the non-uniform hyperbolicity of it. Indeed, they proved that this example is $C^2$-robustly non-uniform hyperbolic (see Section \ref{sec.prel} for the definition of non-uniform hyperbolicity).  

In \cite{obataergodicity}, the author obtained that the Berger-Carrasco example is $C^2$-stably Bernoulli. The technical part of the proof passes through having a precise control of Pesin's stable and unstable manifolds on the fibers (which as we mentioned before is closely related to the standard map). In a certain way, this work goes in the direction of {\em what does the understanding of Berger-Carrasco's example can give us about the standard map?}

Let us explain the strategy of the proof of Theorem \ref{thm.maintheoremuniqueness}. In \cite{buzzicrovisiersarig}, the authors obtain a criterion for uniqueness of the m.m.e. They proved that in a measured homoclinic class (see Definition \ref{def.measuredhomo}) there exists at most one m.m.e. (see Theorem \ref{thm.criteriabcs}). This criterion is based on the powerful tool developed by Sarig in \cite{sarig}, where he obtains a semi-conjugacy of the system with a Markov shift that captures every ``sufficiently hyperbolic'' invariant measure. The Markov shift obtained in \cite{sarig} is not necessarily irreducible. In \cite{buzzicrovisiersarig} they obtain that in a measured homoclinic class one can obtain a semi-conjugacy with a irreducible Markov shift, and this will imply their criterion (see Section 3 in \cite{buzzicrovisiersarig}). With this criterion we reduce the problem of uniqueness of the m.m.e. for the standard map to a problem of finding transverse intersections between stable and unstable manifolds of hyperbolic measures.

We then prove that for sufficiently large parameters any two ergodic measures with ``large'' Lyapunov exponents are homoclinically related (see Theorem \ref{thm.hommeasures}). Using some estimates implied by the result from \cite{duarte} (see Theorem \ref{thm.duartebasicset}), we obtain that for large enough parameters the standard map has high enough topological entropy, so that any two ergodic measures with ``high entropy'' will verify the conditions of Theorem \ref{thm.hommeasures}. In particular, any two measure with ``high entropy'' will be homoclinically related and this will imply the uniqueness of the m.m.e.

The proof of  Theorem \ref{thm.hommeasures} is based on the precise estimates obtained by the author in \cite{obataergodicity} mentioned above. In order to obtain transverse intersections between stable and unstable manifolds, one needs to control the length and ``geometry'' of such manifolds. There is a local and a global strategy in the argument.

For the local strategy, we use the construction of stable manifolds given by Crovisier and Pujals in \cite{ch1crovisierpujalsstronglydissipative}, where in a certain way ``quantifies'' the idea that ``large'' Lyapunov exponents implies ``large'' stable and unstable manifolds (see Proposition \ref{prop.estimateprop1}). This construction together with Pliss lemma (see Lemma \ref{pliss}) will give us some lower bound on the size  and some control of the ``geometry'' of local stable and unstable Pesin's manifolds in a ``large'' set of points for a measure with ``large'' exponents. Only with these estimates one could also conclude the finiteness of m.m.e. (this finiteness is also a consequence of the main theorem in \cite{buzzicrovisiersarig}).

The global strategy goes as follows. Pliss lemma will also give that for a measure with ``large'' exponents, the points obtained before (with precise estimates on the length and ``geometry'' of the stable and unstable manifolds) spend a ``long'' time in the hyperbolic region for the standard map. This will allow us to prove that these stable and unstable manifolds are large ``vertical'' and ``horizontal'' curves in the torus, respectively (see Lemma \ref{bigmnfld}). Then we can find transverse intersections between stable and unstable manifolds for any two ergodic measures with ``large'' exponents. The estimates obtained are $C^2$-robust, and this will imply Theorem \ref{thm.maintheoremuniqueness}.  

\subsection*{Questions and remarks}

As we mentioned before, Newhouse proved that a $C^{\infty}$-diffeomorphism has a m.m.e., see \cite{newhouse}. However, Buzzi showed that for each $r\in [1, +\infty)$ and any surface, there is a $C^r$-diffeomorphism with no m.m.e. (see \cite{buzzi14}).

We remark that our Theorem \ref{thm.maintheoremuniqueness} only addresses the uniqueness of the m.m.e. in a $C^2$-neighborhood of $f_k$, for large $k$. We do not know about the existence of a m.m.e. for any $C^2$-diffeomorphism close to $f_k$. For a diffeomorphism $f$ define
\[
\displaystyle \lambda_{\mathrm{min}}(f) := \min \{ \limsup_{n\to + \infty} \frac{1}{n} \log \|Df^n\|, \limsup_{n\to +\infty} \frac{1}{n} \log \|Df^{-n}\|\},
\]
where $\|Df^n\| := \max_{p\in S} \|Df^n(p)\|$. In \cite{buzzicrovisiersarig}, the authors make the following conjecture.

\begin{conjecture}[Conjecture $2$ in \cite{buzzicrovisiersarig}]
For any $r\in (1, + \infty)$, any $C^r$ diffeomorphism of a surface $f$ with topological entropy larger than $\frac{\lambda_{\mathrm{min}}(f)}{r}$ has a m.m.e.
\end{conjecture}

We remark that $\lambda_{\mathrm{min}}(f) \leq \max \{\log \|Df\|, \log \|Df^{-1}\|\}$. As an easy consequence of Corollary \ref{cor.entropyduarte} below, we can obtain that for $k$ large enough, if $g$ is a $C^2$-diffeomorphism sufficiently $C^1$-close to $f_k$ then $g$ has topological entropy larger than $\frac{\lambda_{\mathrm{min}}(g)}{2}$. Hence, we expect the following conjecture to be true.

\begin{conjecturen}
For $k$ large enough, if $g$ is a $C^2$-diffeomorphism sufficiently $C^2$-close to $f_k$ then there exists a m.m.e. for $g$, which is unique.
\end{conjecturen} 

For a diffeomorphism $f$, an invariant measure $\mu$ is \textbf{exponentially mixing} if there exists $\beta \in (0,1)$, such that for any two H\"older continuous functions $\varphi, \psi$ with zero $\mu$-average\footnote{That is $\int \varphi d\mu = \int \psi d\mu =0$.}, verifies
\[
\displaystyle \left| \int \varphi. \psi \circ f^n d\mu \right| \leq C(\varphi, \psi) \beta^n,
\]
where $C(\varphi, \psi)$ is a constant depending on the functions $\varphi, \psi$.

We remark that some of the features that makes the analysis of the standard map so difficult is that it has expansion and contraction happening in a big part of the manifold, but in some critical regions it may switch expanding and contracting directions. Another important example of surface diffeomorphism with similar properties, but in the dissipative setting, is given by the H\'enon family. Recently, Berger proved in \cite{bergerhenon} that every strongly regular\footnote{See  \cite{bergeryoccoz} for other results on strong regularity.} H\'enon map has a unique m.m.e. He also obtains several properties of this measure, including exponential mixing. A natural question is then the following.

\begin{q}
For $k$ large enough, is the unique m.m.e. for $f_k$ exponentially mixing? Is this property $C^2$-robust?
\end{q}

Our Corollary \ref{cor.continuitymme} states that there exists $k_0$ such that the map $k\mapsto \mu_{\mathrm{max}}(k)$ is continuous, for $k\in [k_0, +\infty)$. This gives a continuous curve in $\mathbb{P}(\T^2)$. 

\begin{q}
What can one say about the regularity of the curve $k\mapsto \mu_{\mathrm{max}}(k)$?
\end{q}

For the measure of maximal entropy $\mu_{\mathrm{max}}(k)$ let $\lambda^+(k)$ be the associated positive Lyapunov exponent. 

\begin{q}
Does there exists $k_0\in \R$ such that the function $k\mapsto \lambda^+(k)$ is continuous for $k\in [k_0,+\infty)$? If so, what is the regularity of this function?
\end{q}

%In \cite{buzzi97}, Buzzi introduced the notion of entropy conjugacy. Two diffeomorphisms $f$ and $g$ are \textbf{entropy conjugated} if there exists a measurable isomorphism identifying measure with ``high enough'' entropy of $f$ with a measure of ``high'' entropy for $g$. 

By the result of Newhouse in \cite{newhouse}, in dimension $2$, the topological entropy depends continuously with the diffeomorphism for the $C^{\infty}$-topology. Hence, the map $k\mapsto h_{\mathrm{top}}(f_k)$ is continuous for $k\in \R$. From the results in \cite{duarte}, we also known that the topological entropy of $f_k$ goes to infinity as $k$ goes to infinity.

\begin{q}
\label{q.increasing}
Does there exist $k_0\in \R$ such that the function $k\mapsto h_{\mathrm{top}}(f_k)$ is strictly increasing for $k\in [k_0, +\infty)$? What is the regularity of this function?
\end{q}

If there was a large parameter $k\in \R$ such that the m.m.e. $\mu$ for $f_k$ verified $\mathrm{dim}_H(\mu) = 2$, then the measure $\mu$ would be an SRB measure, which in the volume preserving case it implies that $\mu$ is absolutely continuous with respect to the Lebesgue measure. This follows from a combination of the dimension formula given in \cite{young82} and Ledrappier-Young's result given in \cite{ledrappieryoung1, ledrappieryoung2}, see also these references for the definition of SRB measure. If this was the case for some large parameter, it would imply positive metric entropy as well. In general, it is not expected that the m.m.e. coincides with an SRB measure. In a certain way our Theorem \ref{thm.maintheoremdensity} states that the m.m.e. for the standard map is ``close'' of being absolutely continuous, meaning the dimension of the measure can be arbitrarily close to $2$.

\begin{q}
Is it true that for any $k$ large enough the m.m.e. $\mu$ for $f_k$ verifies $\mathrm{dim}_H(\mu) < 2$?
\end{q}

As we explained, the main techniques used in this paper are to obtain good estimates of the length and ``geometry'' of stable and unstable Pesin's manifolds for measures with ``large'' exponents. Using these types of estimates, we are able to obtain transverse intersections between those manifolds for any two measures with ``large'' exponents. 

If Sinai's Conjecture \ref{conjecture.sinai} is true (existence of parameters for which the standard map has positive metric entropy for the Lebesgue measure), we believe that our techniques could be used to obtain an upper bound  of the number of ergodic components of the Lebesgue measure restricted to the non-uniformly hyperbolic part.

\subsection*{Organization of the paper}

In Section \ref{sec.prel} we recall the main results from \cite{buzzicrovisiersarig, duarte, gorodetski} that we will use. Sections \ref{intmnfld} and \ref{sec.homrelation} are dedicated to obtain the estimates needed to prove that any two measures with ``large'' exponents are homoclinically related (see Theorem \ref{thm.hommeasures}). Using this we prove Theorem \ref{thm.maintheoremuniqueness} in Section \ref{sec.uniqueness}. Theorem \ref{thm.maintheoremperiodicpoints} is then proved in Section \ref{sec.thmb}, and Theorems \ref{thm.maintheoremdensity} and \ref{thm.genericparamenters} are proved in Section \ref{sec.thmcd}.

\subsection*{Acknowledgments}

This problem was suggested by Pierre Berger on the day of the author's Ph.D. defense. The author would like to thank him for the suggestion. The author also thanks J\'er\^ ome Buzzi, Sylvain Crovisier,  Pedro Duarte, Todd Fisher, Yuri Lima and Mauricio Poletti for many useful comments and suggestions.

\section{Preliminaries}
\label{sec.prel}

\subsection{Homoclinic classes of hyperbolic measures and m.m.e.} \label{subsec.pesintheory}

In this section we recall some of the results of Buzzi, Crovisier and Sarig in \cite{buzzicrovisiersarig}. Their result will be one of the main ingredients in our strategy in proving the uniqueness of m.m.e. for the standard map. Before, let us review some facts about hyperbolic dynamics.

\subsubsection*{Homoclinic classes and horseshoes}

Let $f$ be a $C^1$-diffeomorphism of a compact manifold $M$. A compact invariant set $\Lambda$ is \textbf{hyperbolic} if there is a $Df$-invariant splitting of the tangent space into two directions $T_{\Lambda}M = E^s \oplus E^u$ with the property that there is $N\in \N$ that verifies for any $p\in \Lambda$, 
\[
\|Df^N(p)|_{E^s}\| < \frac{1}{2} \textrm{ and } \|Df^{-N}(p)|_{E^u}\| <\frac{1}{2}.
\]

It is well known that if $\Lambda$ is a hyperbolic set, then for any $q\in \Lambda$, the following sets are $C^1$-immersed submanifolds
\[
\begin{array}{rcl}
W^s(q) & = & \{ x\in M: d(f^n(x), f^n(q)) \to 0, \textrm{ as $n\to +\infty$}\},\\
W^u(q)& = & \{ x\in M: d(f^{-n}(x), f^{-n}(q)) \to 0, \textrm{ as $n\to +\infty$} \}.
\end{array} 
\]
The sets $W^s(q) $ and $W^u(q)$ are called the \textbf{stable} and \textbf{unstable} manifolds of $q$.

A periodic orbit is \textbf{hyperbolic} if it is a hyperbolic set. The set of hyperbolic periodic orbits of $f$ is denoted by $\mathrm{Per}_h(f)$. Given two orbits $\mathcal{O}_1, \mathcal{O}_2 \in \mathrm{Per}_h(f)$ we say that they are \textbf{homoclinically related} if
\[
W^s(\mathcal{O}_1) \pitchfork W^u(\mathcal{O}_2) \neq \emptyset \textrm{ and } W^s(\mathcal{O}_2) \pitchfork W^u(\mathcal{O}_1) \neq \emptyset.
\]
If two orbits, $\mathcal{O}_1$ and $\mathcal{O}_2$, are homoclinically related, we write $\mathcal{O}_1 \sim \mathcal{O}_2$. The \textbf{homoclinic class} of a hyperbolic periodic orbit $\mathcal{O}$ is defined by
\[
\mathrm{HC}(\mathcal{O}) := \overline{ \{\mathcal{O}' \in \mathrm{Per}_h(f): \mathcal{O}' \sim \mathcal{O} \}}.
\]

An $f$-invariant set $\Lambda$ is \textbf{locally maximal} if there exists a neighborhood $U$ of $\Lambda$ such that $\Lambda = \bigcap_{n\in \Z} f^n(U)$. The set $\Lambda$ is \textbf{transitive}, if it contains a dense orbit. A \textbf{basic set} is a transitive, locally maximal and hyperbolic set. A basic set that is totally disconnected is called a \textbf{horseshoe}.

\subsubsection*{Homoclinic classes of hyperbolic measures}

For a diffeomorphism $f$, we say that a set $R$ has full probability if for any $f$-invariant probability measure $\nu$ it is verified that $\nu(R)=1$.  In what follows we state Oseledets Theorem for surface diffeomorphism.

\begin{theorem}[\cite{ch1barreirapesinbook}, Theorems $2.1.1$ and $2.1.2$]
\label{oseledets}
For any $C^1$-diffeomorphism $f:S\to S$ of a compact surface $S$, there is a set $\mathcal{R}$ of full probability, such that the following properties holds:
\begin{enumerate}
\item for any $p\in \mathcal{R}$ there are numbers $s(p) = 1$ or $2$, $\lambda_1(p) <  \lambda_{s(p)}(p)$ and a decomposition $T_pM = E^{1}_p \oplus E^{s(p)}_p$ that verifies
\[
\displaystyle \lim_{n\to +\infty}\frac{1}{n} \log \|Df^n(p)|_{E^i_p}\| = \lambda_i(p), \textrm{ for $i=1$ or $s(p)$;}
\]
\item $s(f(p)) = s(p)$, $\lambda_i(f(p)) = \lambda_i(p)$ and $Df(p).E^{i}_p= E^{i}_{f(p)}$, for every $i= 1, \cdots, s(p)$.
\end{enumerate} 
\end{theorem} 
Notice that if $\mu$ is an ergodic invariant measure for $f$ then the Lyapunov exponents are constant for $\mu$-almost every point.
For $p\in \mathcal{R}$, let 
\[
\displaystyle E^-_p = \bigoplus_{i:\lambda_i(p)<0} E^i_p, \textrm{ and } E^+_p = \bigoplus_{i:\lambda_i(p)>0} E^i_p.
\]

The \textbf{Lyapunov exponents of a periodic point $p$} are the Lyapunov exponents for the invariant ergodic measure $ \frac{1}{\pi(p)} \sum_{j=0}^{\pi(p)} \delta_{f^j(p)}$, where $\pi(p)$ is the period of the point $p$ and $\delta_q$ is the dirac mass on the point $q$.

\begin{definition}
\label{pesinmanifolddef}
For $f$ a $C^{2}$ diffeomorphism the {\bf stable Pesin manifold} of the point $p\in \mathcal{R}$ is
\[
W^-(p) =\{ q\in M: \displaystyle \limsup_{n\to +\infty} \frac{1}{n} \log d(f^n(p), f^n(q)) <0 \}.
\]
Similarly one defines the {\bf unstable Pesin manifold} as
\[
W^+(p) = \{ q\in M: \displaystyle \limsup_{n\to +\infty} \frac{1}{n} \log d(f^{-n}(p), f^{-n}(q)) <0 \}.
\]
\end{definition}

\begin{remark}
Pesin proved that in this setting stable and unstable Pesin manifolds are immersed submanifolds, see section $4$ of \cite{ch1pesin77} for details.
\end{remark}

Let $S$ be a compact surface with no boundary and fix $f$ a $C^2$-diffeomorphism. Let $\mathbb{P}_e(f)$ be the set of ergodic invariant measures for $f$. A measure $\mu \in \mathbb{P}_e(f)$ is \textbf{hyperbolic} if for $\mu$-almost every point all the Lyapunov exponents are non zero. An ergodic hyperbolic measure is of \textbf{saddle type} if almost every point has one positive and one negative exponent. We will denote the set of ergodic hyperbolic measures of saddle type by $\mathbb{P}_h(f)$. For a measure $\mu \in \mathbb{P}_h(f)$, we will write $\lambda^-(\mu,f)$ and $\lambda^+(\mu,f)$ for the negative and positive Lyapunov exponents for $\mu$, respectively.

\begin{definition}
\label{def.homrelatedmeasures}
For two ergodic measures $\mu_1, \mu_2 \in \mathbb{P}_h(f)$, we write $\mu_1 \preceq \mu_2 $ if there are two sets $\Lambda_1, \Lambda_2$ such that $\mu_i(\Lambda_i)>0$ with the property that for any point $(p_1, p_2) \in \Lambda_1 \times \Lambda_2$ we have that $W^-(p_1)$ intersects transversely $W^+(p_2)$.
\end{definition}

\begin{definition}
\label{def.measuredhomo}
Two ergodic measures $\mu_1, \mu_2 \in \mathbb{P}_h(f)$ are \textbf{homoclincally related} if $\mu_1\preceq \mu_2$ and $\mu_2 \preceq \mu_1$. In this case we write $\mu_1 \sim \mu_2$. The set of measures homoclinically related to a measure $\mu$ is called the \textbf{measured homoclinic class} of $\mu$. Let us denote this set by $\mathcal{H}(\mu)$.  
\end{definition}

If $p$ is a hyperbolic periodic point, we will write $p \sim \mu$ if the measure $ \frac{1}{\pi(p)} \sum_{j=0}^{\pi(p)} \delta_{f^j(p)}$ is homoclinically related to the measure $\mu$. 

We refer the reader to section $2.4$ in \cite{buzzicrovisiersarig} for more properties of measured homoclinic classes. In particular, it is shown that the relation $\sim$ is indeed an equivalence relation. So that for any $\nu \in \mathcal{H}(\mu)$ it holds that $\mathcal{H}(\nu) = \mathcal{H}(\mu)$. 

\begin{definition}
\label{def.tophomclass}
The \textbf{topological homoclinic class} of $\mu\in \mathbb{P}_h(f)$ is the set
\[
\mathrm{HC}(\mu):= \overline{\{\mathrm{supp}(\nu): \nu \in \mathcal{H}(\mu)\}}.
\]
\end{definition}
In Corollary $2.14$ in \cite{buzzicrovisiersarig}, the authors proved that there exists a hyperbolic periodic orbit $\mathcal{O}$ which is homoclinically related to $\mu$ and such that
\[
\mathrm{HC}(\mu) = \mathrm{HC}(\mathcal{O}).
\]

One of the main ingredients in this paper is the following result:
\begin{theorem}[Corollary $3.3$ in \cite{buzzicrovisiersarig}]
\label{thm.criteriabcs}
Let $r>1$ and $f$ be a $C^r$-diffeomorphism of a closed surface $S$. Suppose that $\mu$ is an ergodic, hyperbolic, m.m.e. for $f$. Then:
\begin{enumerate}
\item Any ergodic, hyperbolic m.m.e. $\nu$ which is homoclinically related to $\mu$ is equal to $\mu$.
\item the support of $\mu$ is $\mathrm{HC}(\mu) = \mathrm{HC}(\mathcal{O})$, for some hyperbolic periodic orbit $\mathcal{O}$ which is homoclinically related to $\mu$.
\item There exists $l\in \N$ and probability measures $\mu_1, \cdots, \mu_l$ such that $\mu = \frac{1}{l} \sum_{j=1}^l \mu_j$ with the property that:
\begin{itemize}
\item if $j \neq i$ then $\mu_j$ and $\mu_i$ are singular with respect to each other, for $j,i\in \{1, \cdots, l\}$;
\item $f_*(\mu_j) = \mu_{j+1}$, for $j=1, \cdots ,l$ and setting $l+1=1$;
\item each measure $\mu_j$ is $f^l$ invariant, and the system $(f^l, \mu_j)$ is Bernoulli.
\end{itemize}  
\end{enumerate}
\end{theorem}
We remark that the original statement in \cite{buzzicrovisiersarig}, for item (2) of Theorem \ref{thm.criteriabcs} also gives a formula for the number $l$. However, in our application we will not need that. 

One of the consequences of Theorem \ref{thm.criteriabcs} is that for a saddle type hyperbolic measure $\mu$, there exists at most one m.m.e. in $\mathcal{H}(\mu)$.

\subsection{Basic sets for the standard map}
\label{sec.standardback}

Let us just recall some known facts about the standard map that will be used later. Consider the involution $\mathcal{I}:\T^2 \to \T^2$ given by $\mathcal{I}(x,y) = (y,x)$. One may easily check that the following equality holds:
\begin{equation}
\label{eq.inversestandard}
f_k^{-1} = \mathcal{I} \circ f_k \circ \mathcal{I}, \forall k\in \R.
\end{equation}
In other words, by exchanging the $x$ and $y$ coordinates, the system $f_k$ behaves like $f_k^{-1}$. Since
\[
Df_k(x,y) = 
\begin{pmatrix}
2\pi k \cos(2\pi x) + 2 & -1\\
1 & 0
\end{pmatrix},
\]
we have that for $k$ sufficiently large 
\begin{equation}
\label{eq.c1norm}
\frac{1}{4\pi k} < \|Df_k^{-1}\|^{-1} \leq \|Df_k\| < 4\pi k \textrm{ and } \|Df_k^2\| < 5\pi^{2} k.
\end{equation}
Moreover, the estimates \eqref{eq.c1norm} hold in a $C^2$-neighborhood of $f_k$. In \cite{duarte}, Duarte obtains the existence of the following basic sets for the standard map.

\begin{theorem}[Theorem A in \cite{duarte}]
\label{thm.duartebasicset}
There exists $k_0 \in \R$ such that for any $k\in [k_0, +\infty)$ the following is true: there is a basic set for $f_k$, $\Lambda_k$, that verifies:
\begin{itemize}
\item the dynamics of $f_k|_{\Lambda_k}$ is topologically conjugated to a full Bernoulli shift with $2n_k$ symbols, where
\[
\displaystyle \lim_{k\to +\infty} \frac{2n_k}{4k} = 1;
\]

\item the set $\Lambda_k$ is $\frac{4}{k^{\frac{1}{3}}}$-dense in $\T^2$.
\end{itemize}
\end{theorem}

For any $m\in \N$, since the full Bernoulli shift with $m$-symbols has topological entropy $\log m$. We conclude that the basic set $\Lambda_k$ obtained in Theorem \ref{thm.duartebasicset} has topological entropy $h_{\mathrm{top}}(f_k|_{\Lambda_k}) = \log 2n_k$. Since $\Lambda_k$ is a basic set, one may fix a small $C^1$-neighborhood $\mathcal{U}$ of $f_k$ in $\mathrm{Diff}^1(\T^2)$ such that for any $g\in \mathcal{U}$ there is a basic set $\Lambda_k(g)$ that verifies: $g|_{\Lambda_k(g)}$ is topologically conjugated to $f_k|_{\Lambda_k}$; the set $\Lambda_k(g)$ is close to the set $\Lambda_k$ in the Hausdorff distance. In particular, for any $g\in \mathcal{U}_k$ we have $h_{\mathrm{top}}(g) \geq h_{\mathrm{top}}(g|_{\Lambda_k(g)}) = \log 2n_k$, and we may assume that $\Lambda_k(g)$ is $\frac{8}{k^{\frac{1}{3}}}$-dense in $\T^2$. As an immediate consequence of Theorem \ref{thm.duartebasicset}, we obtain:

\begin{corollary}
\label{cor.entropyduarte}
For any $\delta>0$, there exists $k_0\in \R$ such that for any $k\in [k_0, +\infty)$ there exists $\mathcal{U}$ a $C^1$-neighborhood of $f_k$ in $\mathrm{Diff}^1(\T^2)$ with the property that any diffeomorphism $g\in \mathcal{U}$ has a basic set $\Lambda_k(g)$ which is $\frac{8}{k^{\frac{1}{3}}}$-dense in $\T^2$ and with topological entropy greater than $(1-\delta)\log k$.
\end{corollary}

In \cite{duarte}, Duarte also proved that for a generic large $k$, the set $\Lambda_k$ above is accumulated by elliptic islands. Duarte's result was further improved by Gorodetski in \cite{gorodetski}.

\begin{theorem}[Theorem $1$ in \cite{gorodetski}]
\label{thm.gorod}
There exists $k_0 \in \R$ and a dense $G_{\delta}$-set $\mathcal{R} \subset [k_0, + \infty)$, such that for any $k\in \mathcal{R}$ there is an infinite sequence of basic sets
\begin{equation}
\label{eq.gorod}
\Lambda_k^{(0)} \subset \Lambda_k^{(1)} \subset \cdots \subset \Lambda_k^{(n)} \subset \cdots
\end{equation}
 that has the following properties:
 \begin{enumerate}
 \item $\mathrm{dim}_H(\Lambda_k^{(n)}) \to 2$, as $n\to +\infty$;
 \item let $\Omega_k := \overline{\bigcup_{n\in \N} \Lambda_k^{(n)}}$. Then $\Omega_k$ is a transitive invariant set for $f_k$ and $\mathrm{dim}_H(\Omega_k) = 2$;
 \item for any $x\in \Omega_k$ and any $\varepsilon>0$, we have that $\mathrm{dim}_H(B(x,\varepsilon) \cap \Omega_k) = 2$;
 \item each point of $\Omega_k$ is an accumulation point of elliptic islands for $f_k$.
 \end{enumerate}
\end{theorem}

We remark that in the statements of both Duarte and Gorodetski's theorems above, we do not include all the properties that they obtained for these basic sets.

\section{Estimates on invariant manifolds for measures with large exponents}
\label{intmnfld}

The estimates in this section are the equivalent of the estimates made by the author in Section $3$ of \cite{obataergodicity} in the partially hyperbolic setting. The main goal of this section is to prove Proposition \ref{prop.estimateprop1} below. Throughout this section we fix $\delta = \frac{1}{600}$.

\begin{proposition}
\label{prop.estimateprop1}
For $k$ large enough and for any $f_k$-ergodic probability measure $\mu$ such that 
\begin{equation}
\label{eq.measurelargeexponents}
\min \{\lambda^+(\mu,f_k), -\lambda^-(\mu,f_k)\} > (1-\delta) \log k,
\end{equation}
there exists a set with $\mu$-measure larger than $\frac{1-7\delta}{1+7\delta}$, such that:

For any $p$ in that set, there exist a stable and an unstable manifolds at $p$ with length bounded from below by $k^{-7}$. Moreover, the stable manifold is transverse to the horizontal direction and the unstable manifold is transverse to the vertical direction. 
\end{proposition} 

The proof of Proposition \ref{prop.estimateprop1} will follow from Lemma \ref{setzs} and Proposition \ref{sizemnfld} below.

\begin{remark}
From now on, when we refer to $f_k$ we will omit the dependence of $k$ by writing $f= f_k$. 
\end{remark}

We fix two scales $\theta_1= k^{-\frac{2}{5}}$ and $\theta_2 = k^{-\frac{3}{5}}$.
% We also suppose for the rest of the section that $\mu$ is an $f$-ergodic measure verifying \eqref{eq.measurelargeexponents} .  

\subsection{ Points with good contraction and expansion}
Let $\Lambda_{\mu}$ be the set of full $\mu$-measure such that for any point $x\in \Lambda_{\mu}$ we have

\begin{equation}
\label{setlambda}
\displaystyle \frac{1}{n} \sum_{j=0}^{n-1} \delta_{f^j(p)} \xrightarrow[n\to + \infty]{} \mu \textrm{ and } \frac{1}{n} \sum_{j=0}^{n-1} \delta_{f^{-j}(p)} \xrightarrow[n \to +\infty]{} \mu \textrm{, in the $weak^*$-topology.}
\end{equation}
Where $\delta_p$ is the dirac mass on the point $p$. 

Recall that $\mathcal{R}$ is the set of regular points given by Oseledets theorem. Define the sets 
\[
\arraycolsep=1.2pt\def\arraystretch{2}
\begin{array}{rcl}
Z_{\mu}^- & = & \left\{ p\in \mathcal{R} \cap \Lambda_{\mu}: \forall n\geq 0 \textrm{ it holds } \displaystyle \left \Vert Df^n(p)|_{E^{-}_p}\right \Vert < \left(k^{-\frac{4}{5}} \right)^n\right \};\\
Z_{\mu}^+ & = & \left \{ p\in \mathcal{R} \cap \Lambda_{\mu}: \forall n\geq 0 \textrm{ it holds }\displaystyle\left \Vert Df^{-n}(p)|_{E^{+}_p}\right \Vert < \left( k^{-\frac{4}{5}} \right)^n\right \};\\
Z_{\mu} & = & f(Z_{\mu}^-) \cap f^{-1}(Z_{\mu}^+).
\end{array}
\]

\begin{remark}
\label{contaanterior}
By the definition of $Z_{\mu}$, $f^{-1}(Z_{\mu}) \subset Z_{\mu}^-$. Observe that for $p\in Z_{\mu}$ we have 
\[
1\leq  \left \Vert Df(f^{-1}(p))|_{E^-_{f^{-1}(p)}}\right \Vert . \left \Vert Df^{-1}(p)|_{E^-_p}\right \Vert \leq k^{-\frac{4}{5}}\left \Vert Df^{-1}(p)|_{E^-_p}\right \Vert 
\]
We conclude that $\left \Vert Df^{-1}(x)|_{E^-_x}\right \Vert \geq k^{\frac{4}{5}}$. Similarly $\left \Vert Df(x)|_{E^+_x}\right \Vert \geq k^{\frac{4}{5}}$.
\end{remark}

We will need the following version of the Pliss lemma.

\begin{lemma}[ \cite{ch1crovisierpujalsstronglydissipative}, Lemma $3.1$]
\label{pliss}
For any $\varepsilon>0$, $\alpha_1 < \alpha_2$ and any sequence $(a_i) \in (\alpha_1, +\infty)^{\N}$ satisfying
$$
\displaystyle \limsup_{n\to +\infty} \frac{a_0 + \cdots + a_{n-1}}{n} \leq \alpha_2,
$$
there exists a sequence of integers $0 < n_1 < n_2 < \cdots $ such that 
\begin{enumerate}
\item for any $l \geq 1$ and $n>n_l$, one has $\displaystyle \frac{a_{n_l} + \cdots + a_{n-1}}{(n-n_l)} \leq \alpha_2 + \varepsilon$;\\
\item the upper density $\displaystyle \limsup \frac{l}{n_l}$ is larger than $\displaystyle \frac{\varepsilon}{\alpha_2+ \varepsilon - \alpha_1}$.
\end{enumerate}
\end{lemma}

Using this lemma we prove the following.

\begin{lemma}
\label{setzs}
If $k$ is large enough and $\mu$ is an $f$-ergodic measure verifying \eqref{eq.measurelargeexponents}, then $\mu(Z_{\mu}) \geq \frac{1-7\delta}{1+7\delta}$.

\end{lemma}

\begin{proof}
Let $k$ be sufficiently large such that \eqref{eq.c1norm} holds. We may also suppose that $k$ is large enough such that 
\begin{equation}
\label{eq.largek}
\displaystyle \frac{6 \log 4\pi}{\log k} < \delta.
\end{equation}
Let $\mu$ be an $f$-ergodic measure verifying \eqref{eq.measurelargeexponents}. For $p\in \mathcal{R} \cap \Lambda_{\mu}$ and since $E^-_p$ is one dimensional, we obtain
\[
\displaystyle \lim_{n\to +\infty} \frac{1}{n} \log \|Df^n(p)|_{E^{-}_p}\|= \lim_{n\to +\infty}\frac{1}{n} \sum_{j=0}^{n-1} \log \|Df(f^j(p))|_{E^-_{f^j(p)}}\| \leq -(1-\delta) \log k.
\]

Take $\displaystyle \varepsilon = \frac{1}{6}\log k$, $\alpha_1 = -\log k - \log 4\pi$, $\alpha_2 = -(1-\delta) \log k$ and consider the sequence $\left(\log \| Df(f^j(p))|_{E^-_{f^j(p)}}\|\right)_{j\in \N}$. Applying Lemma \ref{pliss} for those quantities we obtain a sequence of integers $(n_l)_{l\in \N}$ such that for every $l\in \N$ and $n> n_l$ 
\[
\displaystyle \frac{1}{n-n_l} \sum_{j=n_l}^{n-1} \log \|Df(f^j(p))|_{E^-_{f^j(p)}}\| \leq -(1-\delta) \log k + \frac{1}{6} \log k = \log k^{-\frac{5}{6}+ \delta} < \log k^{-\frac{4}{5}}.
\]

From this we conclude

\[
\|Df^n(f^{n_l}(p))|_{E^-_{f^{n_l}(p)}}\| < \left( k^{-\frac{4}{5}} \right)^n, \textrm{ } \forall n \geq 0.
\]

Thus for every $l\in \N$ we have $f^{n_l}(p) \in Z_{\mu}^-$. Since $p\in \Lambda_{\mu}$, by Birkhoff's theorem, the estimate \eqref{eq.largek} and the second point in Pliss lemma, we obtain the following estimate 

\[\arraycolsep=1.2pt\def\arraystretch{2}
\begin{array}{rcl}
\displaystyle \mu(Z_{\mu}^-) & \geq & \displaystyle \limsup_{l \to + \infty} \frac{l}{n_l} \\
&\geq &\displaystyle \frac{\varepsilon}{-(1-\delta) \log k + \varepsilon +\log k +\log 4\pi}\\
&= & \displaystyle\frac{1}{ (1+6\delta) + \frac{6\log 4\pi}{\log k}} \geq \frac{1}{1+7\delta}.
\end{array}
\]

Similarly, $\mu(Z_{\mu}^+) \geq \frac{1}{1+ 7\delta}$. This implies that
\[
\mu(\T^2-Z_{\mu}^*) \leq \frac{7\delta}{1+7\delta}, \textrm{ for $* = -,+ $}.
\]

From the definition of $Z_{\mu}$ we conclude that

\begin{equation*}
\mu(Z_{\mu}) = 1- \mu(\T^2-Z_{\mu}) \geq 1- \frac{14\delta}{1+7\delta} = \frac{1-7\delta}{1+7\delta}.\qedhere
\end{equation*} 

\end{proof} 
 
 Let $T=\left[ \frac{1+7\delta}{28\delta}\right]$, since $\delta = \frac{1}{600}$ we have that $T>20$. Define 
\begin{equation}
\label{X}
X_{\mu}=  \displaystyle \bigcap_{j=-T+1}^{T-1} f^j(Z_{\mu}).
\end{equation}

\begin{lemma}
\label{measure}
If $k$ and $\mu$ verify Lemma \ref{setzs}, then $ \mu (X_{\mu}) >0$.
\end{lemma}

\begin{proof}
Recall that $\mu(Z_{\mu})\geq \frac{1-7\delta}{1+7\delta} $, hence 
$$
\mu(\T^2-Z_{\mu}) \leq \frac{14\delta}{1+7\delta}.
$$
Therefore
$$
\begin{array}{rcl}
\mu(X_{\mu})  =  1- \mu(X_{\mu}^c)& \geq & 1- \displaystyle \sum_{j=-T+1}^{T-1} \mu(f^j(\T^2- Z_{\mu})) \\
&\geq & 1- \left(2\left[\frac{(1+7\delta)}{28\delta}\right]-2\right).\frac{14\delta}{1+7\delta} >0.
\end{array}
$$
\end{proof}

\subsection{Cone estimates}
\label{manysets}

Let $V\subset \R^2$ be a one dimensional vector subspace inside $\R^2$ and let $V^{\perp}$ be the one dimensional subspace perpendicular to $V$. For any vector $w\in \R^2$ we can write $w= w_V + w_{V^{\perp}}$, the decomposition of $w$ in $V$ and $V^{\perp}$ coordinates. For $\theta>0$ define 
\[
\C_{\theta}( V) = \{w\in \R^2: \theta \|w_V\| \geq \|w_{V^{\perp}}\|\},
\]
the cone inside $\R^2$ around $V$ of size $\theta$. For simplicity if $V = \R.(1,0)$ then we just write $\C^{hor}_{\theta}= \C_{\theta}(V)$ and  $\C_{\theta}^{ver} = \C_{\theta}(V^{\perp})$, we will call them the horizontal and vertical cones respectively.

Recall that $\theta_1 = k^{-\frac{2}{5}}$. 

\begin{lemma}
\label{cone1}
For $k$ large enough, and $\mu$ an $f_k$-ergodic measure verifying \eqref{eq.measurelargeexponents}, for every $ p\in Z_{\mu}$ we have that $E^+_p \subset \mathscr{C}^{hor}_{\theta_1^{-1}}$, with $\theta_1 = k^{-\frac{2}{5}}$. Furthermore, $\mathscr{C}_{\frac{\theta_1}{2}}(E^+_p) \subset \mathscr{C}^{hor}_{\frac{4}{\theta_1}}$. The same is valid for the $E^-_p$ direction and the vertical cone.
\end{lemma}

\begin{proof}
From Remark \ref{contaanterior}, we know that $\|Df(p)|_{E^+_p}\| \geq k^{\frac{4}{5}}$, for $p=(x,y)\in Z_{\mu}$. Take a vector of the form $(u,1)$, with $|u| \leq k^{-\frac{2}{5}}$, then for $k$ large enough
	\[\arraycolsep=1.2pt\def\arraystretch{1.5}
	\begin{array}{rcl}
	\|Df(p).(u,1)\| &=& |u||2\pi k \cos (2\pi x) +2| + 1 + |u| \\
	&\leq & k^{-\frac{2}{5}}.k^{1 +\frac{1}{200}} + 1\leq k^{\frac{3}{5}+\frac{1}{200}}+1  \leq k^{\frac{3}{5}+\frac{1}{100}} <k^{\frac{4}{5}}.
	\end{array}
	\]
		
	Hence, if $p\in Z_{\mu}$ then $E^+_p \subset \mathscr{C}^{hor}_{\theta_1^{-1}}$.
	
We want to determine $\theta>0$ such that the cone $\C^{hor}_{\theta}$ contains the cone $\C_{\frac{\theta_1}{2}}(E^+_x)$. For this purpose we will consider a cone $\C_{\frac{\theta_1}{2}}(V)$, where the direction $V$ belongs to the boundary of the cone $\C^{hor}_{\theta_1^{-1}}$.

Suppose $V$ is generated by the unit vector $(u,\frac{u}{\theta_1})$, with $u>0$. Observe that $V^{\perp}$ is generated by $(-\frac{u}{\theta_1},u)$. One of the boundaries of the cone $\C^{hor}_{\theta}$ we are looking for is generated by the vector $\frac{\theta_1}{2}(-\frac{u}{\theta_1},u) + (u,\frac{u}{\theta_1})$.
	
	The size of the cone $\theta$ is given by 
	$$
	\theta = \frac{2.[u(\theta_1^2+2)]}{2u\theta_1} =\frac{\theta_1^2+2}{\theta_1} <\frac{4}{\theta_1}.
	$$
	
	Since the horizontal cones are symmetric with respect to the horizontal direction, we conclude that 
	\[
	\C_{\frac{\theta_1}{2}}(E^+_p) \subset\C^{hor}_{\theta} \subsetneq \C^{hor}_{\frac{4}{\theta_1}}.
	\]
	
	By \eqref{eq.inversestandard}, a similar argument holds for the stable direction but using vertical cones.
\end{proof}

We define some critical regions. For that, define $I_1=I_1(k)= (-2k^{-\frac{3}{10}},2 k^{-\frac{3}{10}})$, $I_2=I_2(k) = \frac{I_1}{2}$, write $C_1 = \{\frac{1}{4}+ I_1\} \cup \{\frac{3}{4} +I_1\}$ and $C_2 =\{\frac{1}{4}+I_2\} \cup \{\frac{3}{4} +I_2\}$. Consider the regions
\[\arraycolsep=1.2pt\def\arraystretch{1.5}
\begin{array}{rclclcl}
\displaystyle Crit_1 &=& \{ C_1 \times S^1  \} \cup \{ S^1 \times C_1 \} & \textrm{ and } & Crit_2 &=& \{ C_2 \times S^1   \} \cup \{ S^1 \times C_2  \}.
\end{array}
\]

Write $G_* = (Crit_*)^c$, for $*=1,2$ and observe that $G_1 \subset G_2$. Observe also that each $G_*$ has four connected components, $\{G_{*,j}\}_{j=1}^4$. Each $G_{*,j}$ is a square and we can choose the index $j$ such that $G_{1,j} \subset G_{2,j}$.

\begin{remark}
\label{remarkado}
  The distance between the boundaries of these two sets is  
$$
d(\partial G_{1,j}, \partial G_{2,j}) = k^{-\frac{3}{10}} > k^{-7}, \textrm{ for $1 \leq j \leq 4$.}
$$
\end{remark}

Recall that $\theta_2= k^{-\frac{3}{5}}$.

\begin{lemma}
\label{esqueci}
If $k$ is large enough then for any $f_k$-ergodic measure $\mu$ verifying \eqref{eq.measurelargeexponents}, we have
\begin{enumerate}
\item $Z_{\mu} \subset G_1 \subset G_2;$
\item If $p\in G_2$ then $Df(p).(\mathscr{C}^{hor}_{\frac{4}{\theta_1}}) \subset \mathscr{C}^{hor}_{\theta_2}$;
\item If $\gamma$ is a $C^1$-curve with length $l(\gamma)\geq k^{-\frac{3}{10}}$, such that $\gamma \subset G_2$ and it is tangent to $\mathscr{C}^{hor}_{\theta_2}$ then $l(f(\gamma)) >4$.
\end{enumerate}
Similar statements hold for the vertical cone and $f^{-1}$.
\end{lemma}

\begin{proof}

1. If $p=(x,y)\notin G_1$ then for $k$ large enough, $|\cos ( 2\pi x)| < 4k^{-\frac{3}{10}}$, in particular
\[
\|Df(p)\| \leq 2\pi k|\cos (2\pi x)| + 4 < 8\pi k^{\frac{7}{10}} + 4 <k^{\frac{4}{5}}.
\]
A similar calculation implies that for $p\notin G_1$ we have 
\[
\|Df^{-1}(p)\| < k^{\frac{4}{5}}.
\]
Thus $Z_{\mu} \subset G_1 \subset G_2.$
\begin{enumerate}
\setcounter{enumi}{1}
\item For any $p=(x,y)\in G_2$, observe that  
\begin{equation}
\label{cosx}
|\cos (2\pi x)| \geq \frac{k^{-\frac{3}{10}}}{2}.
\end{equation}

Hence, for any vector $(u,v) \in \mathscr{C}^{hor}_{\frac{4}{\theta_1}}$ we have
\[
\theta_2 ( |2\pi k \cos(2\pi x) +2||u| - |v|) \geq \theta_2|u| \left( \pi.k^{\frac{7}{10}} - 2 - 4k^{\frac{3}{5}}\right)= |u|\left( \pi k^{\frac{1}{10}} - 2k^{-\frac{3}{5}} -4 \right) >|u|.
\]

\item For any $p=(x,y)\in G_2$, for $(u,v) \in \mathscr{C}^{hor}_{\theta_2}$ an unit vector, we must have
\[
\arraycolsep=1.2pt\def\arraystretch{1.5}
\begin{array}{rclcl}
\|Df(p).(u,v)\| & \geq & |2\pi \cos(2\pi x) + 2||u| - |v| & \geq & |u| ( |2\pi \cos(2\pi x) + 2| -\theta_2)\\
& \geq & \frac{\|(u,v)\|}{1+\theta_2} (|2\pi \cos(2\pi x) + 2| -\theta_2) & \geq &\frac{1}{2}( 2 \pi k|\cos x| -2 - \theta_2) \\
& \geq & \frac{\pi k^{\frac{7}{10}}}{2} - 1- \frac{\theta_2}{2} &> &  k^{\frac{1}{2}}.
\end{array}
\]
Hence,
\begin{equation*}
l(f(\gamma)) \geq k^{\frac{1}{2}}.k^{-\frac{3}{10}} = k^{\frac{2}{10}} >4. \qedhere
\end{equation*}
\end{enumerate}
\end{proof}

\begin{remark}
\label{remark.robust1}
Notice that the proofs of the Lemmas \ref{setzs}, \ref{measure}, \ref{cone1} and \ref{esqueci} above are actually $C^1$-robust. That is, if $k$ is large enough, and $g$ is sufficiently $C^1$-close to $f_k$, any $g$-ergodic measure $\mu$ verifying \eqref{eq.measurelargeexponents} verify the conclusion of these lemmas.
\end{remark}

\subsection{A lower bound on the size of the invariant manifolds}
\label{lowerbound}
The next proposition gives us the existence of stable and unstable manifolds with good estimates on its sizes and its tangent directions. The proof of this proposition follows the exact same steps as Theorem $5$ in \cite{ch1crovisierpujalsstronglydissipative}, with the adaptations needed in order to obtain the precise estimates we will use. This was also done previously in the partially hyperbolic setting by the author in \cite{obataergodicity}.

Theorem $5$ in \cite{ch1crovisierpujalsstronglydissipative} proves the existence of stable manifolds with uniform size and ``geometry" in the following scenario. Let $g:S \to S$ be a $C^2$-diffeomorphism of a compact surface and let $\sigma, \tilde{\sigma}, \rho, \tilde{\rho} \in (0,1)$ be constants such that 
\begin{equation}
\label{inequalityimp}
\frac{\tilde{\sigma} \tilde{\rho}}{\sigma \rho} > \sigma.
\end{equation}
Let $p\in S$  and let $E \subset T_pS$ be a direction such that for all $n\geq 0$
\[
\tilde{\sigma}^n \leq \|Dg^n(p)|_{E}\| \leq \sigma^n \textrm{ and } \tilde{\rho}^n \leq \frac{\|Dg^n(p)|_E\|^2}{|\det Dg^n(p)|} \leq \rho^n.
\]  
Thorem $5$ in \cite{ch1crovisierpujalsstronglydissipative} then gives the existence of a stable manifold on $p$ tangent to the direction $E$ whose size depends only on the constants $\sigma, \tilde{\sigma}, \rho, \tilde{\rho}, \|f\|_{C^2}$. 

We remark that inequality (\ref{inequalityimp}) is important in the construction. That is the main reason why in this paper we will work with hyperbolic measures of ``large'' exponents.

\begin{proposition}
\label{sizemnfld}
For $k$ large enough, for any $f$-ergodic measure $\mu$ verifying \eqref{eq.measurelargeexponents}, for each $p\in Z_{\mu}$, there are two $C^1$-curves $W^*(p)$ tangent to $E^*_p$ and with length bounded from below by $r_0 = k^{-7}$, for $*=-,+$. Those curves are $C^1$-stable and unstable manifolds for $f$, respectively. Moreover, $T_qW^+_{r_0}(p) \subset \mathscr{C}^{hor}_{\frac{4}{\theta_1}}$ and  $T_mW^-_{r_0}(p) \subset \mathscr{C}^{ver}_{\frac{4}{\theta_1}}$, for every $q\in W^+_{r_0}(p)$ and $m\in W^-_{r_0}(p)$.
\end{proposition}

\begin{proof}
We use some of the notation of the proof of Theorem 5 in \cite{ch1crovisierpujalsstronglydissipative}. If $p\in Z_{\mu}$, since $Z_{\mu} = f(Z_{\mu}^-) \cap f^{-1}(Z_{\mu}^+)$ we have that $f^{-1}(p)\in Z^-_{\mu}$. For such point $p$ it holds that
\[
(4\pi k)^{-n}\leq \left \Vert Df^n(f^{-1}(p))|_{E^-_{f^{-1}(p)}}\right \Vert < \left(k^{-\frac{4}{5}}\right)^n, \textrm{ } \forall n\geq 0.
\]
Since $\left \vert \det Df(p)\right \vert = 1$ for every $p\in \T^2$, it also holds
\[
(4\pi k)^{-2n}\leq \frac{\left \Vert Df^n(f^{-1}(p))|_{E^-_{f^{-1}(p)}}\right \Vert ^2}{\left|\det Df^n(f^{-1}(p))\right|} <\left(k^{-2.\left(\frac{4}{5}\right)}\right)^n, \textrm{ } \forall n\geq 0.
\]

For each $n\in \N$ consider $\psi_n:V_{n} \to T_{f^{n}(p)}\T^2$ to be the lifted dynamics by the exponential map of $f$ along the orbit of $p$, that goes from a neighborhood $V_n$ of $0$ in $T_{f^{n-1}(p)}\T^2$ to a neighborhood of $0$ in $T_{f^{n}(p)}\T^2$. Since $f$ is a $C^2$-diffeomorphism, this implies that $\psi_n$ is a $C^2$-diffeomorphism into its image.

Take $\sigma=k^{-\frac{4}{5}}$, $\tilde{\sigma} = (4\pi k)^{-1}$, $\rho = \sigma^2$ and $\tilde{\rho} = \tilde{\sigma}^2$. Consider
$$
\lambda_1 = 2\sigma \textrm{ and } \lambda_2 = \frac{\tilde{\rho}}{2},
$$
and take
$$
C_0=3 > \displaystyle \sum_{k\geq 0 }\left(\frac{\sigma}{\lambda_1}\right)^k = 2 = \displaystyle \sum_{k\geq 0} \left(\frac{\lambda_2}{\tilde{\rho}}\right)^k.
$$

Let $E_n = E^-_{f^{n-1}(p)}$ and $F_n=E_n^{\perp}$ and use the basis $E_n \oplus F_n$. We define
$$
m_n = \left \Vert Df^n(f^{-1}(p))|_{E^-_{f^{-1}(p))}}\right \Vert \textrm{ and } M_n = \frac{|\det Df^n(f^{-1}(p)|}{m_n} = \frac{1}{m_n}.
$$ 
Using this notation it is also defined
\[\arraycolsep=1.2pt\def\arraystretch{2.5}
\begin{array}{c}
\displaystyle A_n= \sum_{k\geq 0} \lambda_1^{-k} m_{n+k}/m_n,\\
\displaystyle B_n = \sum_{k=0}^n \lambda_2^{k-n} \frac{M_k/M_n}{m_k/m_n}.
\end{array}
\]

The proof of theorem $5$ in \cite{ch1crovisierpujalsstronglydissipative} gives
\begin{equation}
\label{anbn}
A_n \leq C_0 \left(\frac{\lambda_1}{\tilde{\sigma}}\right)^n \textrm{ and } B_n \leq C_0 \left(\frac{\rho}{\lambda_2}\right)^n.
\end{equation}

Define the change of coordinates in $T_{f^{n-1}(p)}\T^2$ given by $\Delta_n = Diag(A_n, A_n B_n)$, where the map $\Delta_n$ is defined using the coordinates $E_n \oplus F_n$. Observe that $A_n$ and $B_n$ are larger or equal to $1$, in particular, $\|\Delta_n\| = A_n B_n$ and $\|\Delta_n^{-1}\|=A_n^{-1}<1$.

Write $h_n= \Delta_{n+1} \circ \psi_n \circ \Delta_n^{-1}$ and $H_n = \Delta_{n+1} \circ D\psi_n(0) \circ \Delta_n^{-1}$. We have
$$
H_n = \begin{pmatrix}
a & d\\
0 & c
\end{pmatrix}
\textrm{ and } 
H_n^{-1} = \begin{pmatrix}
\frac{1}{a} & -\frac{d}{ca}\\
0 & \frac{1}{c}
\end{pmatrix}.
$$

From the proof of Theorem $5$ in \cite{ch1crovisierpujalsstronglydissipative}, we obtain
%\[\arraycolsep=1.2pt\def\arraystretch{1.5}

\begin{alignat}{3}
(\|Df\|.\|Df^{-1}\|^2)^{-1} & \leq & |a| & <&& \lambda_1                  \label{uno}\\
|a|\lambda_2^{-1} & \leq & |c| & \leq && \lambda_1 \lambda_2^{-1} \|Df\|.\|Df^{-1}\| + \lambda_1 \|Df^{-1}\|^2\label{dos}\\
&&|d| &\leq && \|Df\|.\|Df^{-1}\||a|.\label{tres}
\end{alignat}

Using inequalities (\ref{dos}) and (\ref{tres}), we have 
$$
\left|\frac{d}{c}\right|\leq \frac{\|Df\|.\|Df^{-1}\| |a|}{|a| \lambda_2^{-1}} < \frac{(4\pi k)^2}{2.(4\pi k)^2} = \frac{1}{2}.
$$

Let us set $\xi= \frac{\tilde{\sigma} \lambda_2}{\lambda_1^2 \rho}$ and observe that for $k$ large enough $\xi>4$. For $\eta \leq \frac{1}{2}$ we will consider $\widetilde{\mathscr{C}}_{(\eta,n)}= \C_{\eta}(E_n)$ the cone of size $\eta$ around the direction $E_n$. If $(u,v)\in \widetilde{\mathscr{C}}_{(\eta,n+1)}$, using (\ref{uno}) and the estimate on $\left| \frac{d}{c}\right|$, we have
\[\arraycolsep=1.2pt\def\arraystretch{1.5}
\begin{array}{rclcl}
\|H_n^{-1}.(u,v)\| &\geq & \left|\frac{u}{a}\right|- \left|\frac{dv}{ca}\right|& \geq & \left|\frac{u}{a}\right| \left( 1-\left|\frac{d\eta}{c}\right| \right)  \\
& \geq & \frac{\|(u,v)\|}{(1+\eta) \lambda_1} \left( 1- \frac{\eta}{2}\right) & \geq & \frac{\|(u,v)\|}{\frac{3}{2} \lambda_1} . \frac{1}{2} .\frac{3}{2} = \frac{\|(u,v)\|}{2\lambda_1} > \frac{\|(u,v)\|}{\xi \lambda_1}.
\end{array}
\]

We conclude that the vectors of the cone $\widetilde{\mathscr{C}}_{(\eta,n+1)}$ are expanded by $\frac{1}{2 \lambda_1}$ by $H_n^{-1}$. Observe that if a linear map is $\frac{\eta}{6}$-close to $H_n^{-1}$ then the vectors inside $\widetilde{\C}_{\eta,n+1}$ are expanded by at least $(4\lambda_1)^{-1}> (\xi \lambda_1)^{-1}$. It is easy to see that $L \left(\widetilde{\mathscr{C}}_{(\eta, n+1)}\right) \subset \widetilde{\mathscr{C}}_{(\eta, n)}$ for any linear map $L$ which is $\frac{\eta}{6}$-close to $H_n^{-1}$.

Recall that 
\begin{equation}
\label{eq.c1c2}
\|Df\| < 4\pi k \textrm{ and } \|D^2f^{-1}\|< 5\pi^2 k.
\end{equation}
Since $\|\Delta_{n+1}^{-1}\|<1$, we obtain
\[
\|Dh_n^{-1}(0)- Dh_n^{-1}(y)\|\leq \|\Delta_n\|.\|\Delta_{n+1}^{-1}\|.\|D^2f^{-1}\|.\|\Delta_{n+1}^{-1}\|.\|y\| \leq 5\pi^2 k A_nB_n\|y\|.
\]
Using (\ref{anbn}), we have that $Dh^{-1}_n(y)$ is $\frac{\eta}{4|a|}$-close to $H_n^{-1}$ in a ball of radius
$$
\tilde{r}_{n+1} = \frac{\eta}{6 (5\pi^2 k) A_nB_n}> \frac{\eta}{30 \pi^2 k C_0^2  } \left(\frac{\tilde{\sigma}\lambda_2}{\lambda_1 \rho}\right)^{n}>\frac{\eta}{270 \pi^2 k}.(4\lambda_1)^{n}.
$$

Since $Dh_n^{-1}$ expands the vectors inside the cone $\widetilde{\C}_{\eta,n+1}$ by at least $(4\lambda_1)^{-1}> (\xi \lambda)^{-1}$, we can take 
\[
\tilde{r}_0 = \frac{\eta}{300 \pi^2 k}.\frac{1}{4 \lambda_1} =\frac{\eta}{1200 \pi^2 k \lambda_1}. 
\]

The proof of Theorem $5$ in \cite{ch1crovisierpujalsstronglydissipative} gives us a $C^1$-curve inside $T_{f^{-1}(p)}\T^2$ tangent to the cone $\widetilde{\C}_{\eta,0}$, of size $\tilde{r}_0$, which is a stable manifold for the sequence $(h_n)_{n\in \N}$. 

To obtain a stable manifold for the sequence $(\psi_n)_{n\in \N}$ we need to apply $\Delta_0$ to this curve. Recall that $\Delta_0 = Diag(A_0,A_0)$, in particular it preserves the size and direction of a cone. Thus, we obtain that $\Delta_0(\widetilde{\C}_{(\eta,0)})= \C_{\eta}(E^-_{f^{-1}(p)})$.

To obtain a stable manifold for $f$, instead of the sequence $(\psi_n)_{n\in \N}$, we must project this curve by the exponential map, this projection will be denoted by $W^-(f^{-1}(p))$. Since $\T^2$ is the flat torus, the derivative of the exponential map is the identity. We conclude that the stable manifold for $f$ at the point $f^{-1}(p)$ is tangent to $\C_{\eta}(E^-_{f^{-1}(p)})$.

Now we estimate the size of the cones in the proposition at the point $p$. So far, the only restriction we have is $\eta \leq \frac{1}{2}$. Since $\|Df^{-1}\|$ and $\|Df\|$ are bounded from above by $4\pi k$, 
$$
Df(f^{-1}(k)).\C_{\eta}(E^-_{f^{-1}(p)}) \subset \C_{16\pi^2 k^2 \eta}(E^-_p).
$$

Using the estimates from Lemma \ref{cone1}, we want $16 \pi^2 k^2 \eta \leq \frac{\theta_1}{2} = \left(2k^{\frac{2}{5}}\right)^{-1}$, therefore, the additional restriction we put now is $\eta < \left(32 \pi^2 k^{2+\frac{2}{5}}\right)^{-1}$. Since $k$ is large, we can take $\eta= k^{-3}$, for instance. By Lemma \ref{cone1}, we have $E^-_p \subset \C^{ver}_{\theta_1^{-1}}$ and  $\C_{16\pi^2 k^2\eta}(E^-_p) \subset \C^{ver}_{\frac{4}{\theta_1}}$. This proves the estimate on the cones of the proposition. 

With this restriction, now we estimate the size of the stable manifold at the point $p$. For $\eta= k^{-3}$ and since $\lambda_1 = 2k^{-\frac{4}{5}}$, we obtain for $k$ large enough, 
$$
\tilde{r}_0= \frac{\eta}{1200 \pi^2 k\lambda_1} = \frac{1}{2400 k^{4 - \frac{4}{5}}}>\frac{1}{k^5}.
$$

From this one can conclude that the stable manifold at the point $f^{-1}(p)$ has size bounded below by $k^{-5}$, this implies that at the point $p$ the stable manifold has size bounded by $(4\pi k)^{-1}.k^{-5} > k^{-7} =r_0$, which concludes the proof for $W^-_{r_0}(p)$. The proof for the unstable manifold is analogous.
\end{proof}

\begin{remark}
From item $1$ of Lemma \ref{esqueci} and Remark \ref{remarkado}, if $p\in Z_{\mu}$ then $W^*_{r_0}(p) \subset G_2$, for $*= -,+$.
\end{remark}

\begin{remark}
\label{remark.robust2}
Since the estimates \eqref{eq.c1c2} are $C^2$-open, the conclusion of Proposition \ref{prop.estimateprop1} is $C^2$-robust.
\end{remark}

\section{Homoclinic relation for measures with large exponents} \label{sec.homrelation}

Recall that we fixed $\delta = \frac{1}{600}$. The goal of this section is to prove the following theorem:
\begin{theorem}
\label{thm.hommeasures}
There exists $k_0\in \N$ such that for any $k\in [k_0,+\infty)$ the following holds true: if $\mu$ and $\nu$ are two $f_k$-ergodic probability measures that verify 
\[
\min\{\lambda^+(\mu,f), \lambda^+(\nu,f), - \lambda^-(\mu,f), - \lambda^-(\nu,f)\}> (1-\delta) \log k,
\]
then $\mu$ and $\nu$ are homoclinically related. Furthermore, for any $l\in \N$, let $\mu_1, \cdots, \mu_{j}$ and $\nu_1, \cdots, \nu_s$ be the ergodic decomposition of $(f^l,\mu)$ and $(f^l,\nu)$, respectively. Then, any two measures $m_1, m_2 \in \{\mu_1, \cdots, \mu_j, \nu_1, \cdots, \nu_s\}$ are homoclinically related for $f^l$ as well. 
\end{theorem}

Let $k$ be large enough such that $f_k$ verifies the results in Section \ref{intmnfld}. Let $\mu$ be any $f_k$-ergodic measure that verifies the hypothesis of Theorem \ref{thm.hommeasures}. Recall that in \eqref{X} we defined 
\[
X_{\mu} = \bigcup_{j = -T + 1}^{T-1} f^j(Z_{\mu}), \textrm{ where $T = \left[\frac{1+7\delta}{28\delta}\right].$}
\]
 Recall also that  $\theta_2 = k^{-\frac{3}{5}}$ and that we defined in Section \ref{manysets} the sets $G_1$ and $G_2$.

\begin{lemma}
\label{bigmnfld}
For $k$ large enough the following holds: for any $\mu$ ergodic measure verifying \eqref{eq.measurelargeexponents} and any $n\geq 15$, for every $p\in X_{\mu}$ there are two curves $\gamma^-_{-n}(p) \subset f^{-n}(W^-_{r_0}(p))$ and $\gamma^+_n(p) \subset f^n(W^+_{r_0}(p))$ with length greater than $4$. The tangent spaces of each of those curves are contained in the cone $\C^{ver}_{\theta_2}$ and $\C^{hor}_{\theta_2}$, respectively.
\end{lemma}

\begin{proof}

If $p\in X_{\mu}$ then by the definition of $X_{\mu}$ and Lemma \ref{esqueci},
\[
\{f^{-T+1}(p), \cdots , f^{T-1}(p)\}\subset Z_{\mu} \subset G_1\subset G_2, \textrm{ where $T=\left[ \frac{1+7\delta}{28\delta}\right]>20$.}
\]

 Define $W^+_k(p) = f^k(W^+_{r_0}(p))$ and observe that for every $q\in W^+_k(p)$, if $q\in G_2$ and $T_qW^+_k(p)\subset \C^{hor}_{\theta_2}$ then $T_{f(q)}W^+_{k+1}(p) \subset \C^{hor}_{\theta_2}$.

By Proposition \ref{sizemnfld}, $TW^+_0(p) \subset \C^{hor}_{\frac{4}{\theta_1}}$. Since $p\in Z_{\mu} \subset G_1$, by Remark \ref{remarkado} we conclude that $W^+_0(p) \subset G_2$. Item $2$ of Lemma \ref{esqueci} implies that $TW_1^+(p) \subset \C^{hor}_{\theta_2}$. 
 
If $p\in G_2$ and $(u,v) \in \C^{hor}_{\theta_2}$ is an unit vector, then $\|Df(p).(u,v)\|> k^{\frac{1}{2}}$. For a $C^1$-curve $\gamma$ containing $p$ with length $k^{-7}$, such that $ \gamma \subset G_2$ and $T\gamma \subset \C^{hor}_{\theta_2}$, let $m\in \N$ be the largest number such that $f^j(\gamma) \subset G_2$, for every $j=1, \cdots ,m$. Since the vectors inside $\C^{hor}_{\theta_2}$ are expanded by at least $k^{\frac{1}{2}}$ and the cone $\C^{hor}_{\theta_2}$ is preserved by the derivative of the points in $G_2$, we conclude that $m\leq 14$.   

Let $m^+_0 \in \N$ be the smallest number such that $W^+_{k^+_0}(p) \cap \partial G_2 \neq \emptyset$. Recall that if $q\in G_2$ and $(u,v) \in \C^{hor}_{\frac{4}{\theta_1}}$ is a unit vector, then by (\ref{cosx}), $\|Df(q).(u,v)\| >1$. Since $r_0 = k^{-7}$, we obtain that the curve $W_1^+(p) \subset \C^{hor}_{\theta_2}$ has length at least $k^{-7}$ and is tangent to $\C^{hor}_{\theta_2}$, by the previous paragraph $m^+_0 \leq 15$.

If $p\in X_{\mu}$, the connected component of $W^+_{m^+_0}(p)\cap G_2$ containing $f^{m^+_0}(p)$, which we will denote by $\widehat{W}^+_{m^+_0}(p)$, intersects the boundary of $G_2$ and $T\widehat{W}^+_{m^+_0}(p) \subset \C^{hor}_{\theta_2}$. Since $m_0^+ < T$, we know that $f^{m^+_0}(p) \in Z_{\mu} \subset G_1 \subset G_2$. We conclude that $\widehat{W}^+_{m^+_0}(p)$ also intersects the boundary of $G_1$.

 Let $\gamma^+_{m^+_0}$ be a connected component of $\widehat{W}^+_{m^+_0}(p) \cap (G_2-G_1)$, such that $\gamma^+_{m^+_0} \cap \partial G_1 \neq \emptyset$ and $\gamma^+_{m^+_0}\cap \partial G_2 \neq \emptyset$. The curve $\gamma^+_{m^+_0}$ is a $C^1$-curve that verifies the hypothesis of item $3$ from Lemma \ref{esqueci}. Thus $l(f(\gamma^+_{m^+_0})) >4$, $Tf(\gamma^+_{m^+_0}) \subset \C^{hor}_{\theta_2}$ and by definition $f(\gamma^+_{m^+_0}) \subset W^+_{m^+_0+1}(p)$. Define $\gamma_{m_0^+ + 1}(p) = f(\gamma^+_{m_0^+})$.
 
%\begin{figure}[h]
%\centering
%\includegraphics[scale= 0.6]{image1}
%\caption{The curve $\gamma_{k_0^+}^+$}
%\label{figure1}
%\end{figure}

Let 
\[\widetilde{G} = \left\{(x,y) \in \T^4: k^{-\frac{3}{10}} \leq |x - \frac{1}{4}| \leq 2k^{-\frac{3}{10}} \textrm{ or } k^{-\frac{3}{10}} \leq |x-\frac{3}{4}| \leq 2k^{-\frac{3}{10}}\right\}.
\] 
It is easy to see that $\widetilde{G}$ has four connected components, each connected component having two boundaries. Since the critical region only depends on the coordinate $x$, for any point $q \in \widetilde{G}$, the derivative $Df(q)$ expands any vector inside $\C^{hor}_{\theta_2}$ by at least $k^{\frac{1}{2}}$.
 
We build $\gamma_n^+ \subset f(\gamma_{n-1}^+)$ inductively for $n> m_0^+ + 1$. Let us build it for $n= m_0^++2$. Observe that $P_1(\gamma_{m^+_0+1}^+)= S^1$, where $P_1$ is the projection on the first coordinate of the torus with coordinates $(x,y)$. Consider then $\widetilde \gamma^+_{m_0^++1}$ to be a connected component of $\gamma_{m_0^++1}^+(p) \cap \widetilde{G}$ that intersects the two boundaries of a connected component of $\widetilde{G}$. Define $\gamma^+_{m_0^++2}(p) = f(\tilde{\gamma}^+_{m_0^++1})$, observe that $l(\gamma^+_{m_0^++2}(p)) >4$ and $Tf(\gamma_{m_0^++2}(p)) \subset \C^{hor}_{\theta_2}$. In this way we can build inductively the curves $\gamma_n^+(p)$ that verify the conclusions of the lemma. In a similar way, we construct the curves $\gamma^-_{-n}(p)$. Since $m_0^+ \leq 15$ and $m_0^- \leq 15$, then this certainly holds for $n> 15$.
\end{proof}

\begin{proof}[Proof of Theorem \ref{thm.hommeasures}]

The proof of Theorem \ref{thm.hommeasures} will follow from Lemma \ref{bigmnfld}. Let $\mu$ and $\nu$ be as in the hypothesis of Theorem \ref{thm.hommeasures}. Let $X_{\mu}$ and $X_{\nu}$ be the sets defined above.

By Lemma \ref{measure}, $\mu(X_{\mu}),\nu(X_{\nu})>0$. Fix $p_{\mu}\in X_{\mu}$ and $p_{\nu} \in X_{\nu}$. If $n>15$, then by Lemma \ref{bigmnfld} we conclude that $f^n(W^+_{r_0}(p_{\mu})) \pitchfork f^{-n}(W^-_{r_0}(p_{\nu})) \neq \emptyset$. Observe that by the same reason we have that $f^n(W^+_{r_0}(p_{\nu})) \pitchfork f^{-n}(W^-_{r_0}(p_{\mu})) \neq \emptyset$. Hence, the measures $\mu$ and $\nu$ are homoclinically related. 

Fix $l\in \N$. Let $\mu_1, \cdots, \mu_j$ and $\nu_1, \cdots, \nu_s$ be the ergodic decomposition of $(f^l,\mu)$ and $(f^l,\nu)$. Write $X = X_{\mu} \cup X_{\nu}$, and observe that for any two measures $m_1, m_2 \in \{\mu_1, \cdots, \mu_j, \nu_1, \cdots, \nu_s\}$ we have that $m_i(X)>0$, for $i=1$ or $2$.

Since Lemma \ref{bigmnfld} holds for every $n$, and since Pesin's stable and unstable manifolds for $f$ coincides with the manifolds for $f^l$, from the argument above one may also conclude that $m_1$ is homoclinically related to $m_2$ for $f^l$. 
\end{proof}

\section{Proof of Theorem \ref{thm.maintheoremuniqueness}} \label{sec.uniqueness}

Let us first prove the Theorem \ref{thm.maintheoremuniqueness} for the standard map, then we explain why this proof is $C^2$-robust. Recall that we fixed $\delta = \frac{1}{600}$.

\subsection*{Uniqueness of the m.m.e. for the standard map}

Let $k_0\in \R$ be large enough such that for any $k\geq k_0$ the diffeomorphism $f_k: \T^2 \to \T^2$ verifies Theorem \ref{thm.hommeasures}. Furthermore, by Corollary \ref{cor.entropyduarte}, we may also assume that $k_0$ is large enough such that any for any $k\geq k_0$ we have that $h_{top}(f_k) > (1-\delta) \log k$. 

Since $f_k$ is $C^{\infty}$ by Newhouse's result \cite{newhouse}, there is at least one ergodic m.m.e. $\mu$ for $f_k$. Define $\mathbb{P}_{\mathrm{high}}(f_k)$ to be the set of ergodic measures with entropy larger than $(1-\delta)\log k$.

By Ruelle's inequality (see for instance Theorem $5.4.1$ in \cite{ch1barreirapesinbook}), if $\mu$ is an ergodic measure then
\[
h_{\mu}(f_k)\leq \min\{\lambda^+(\mu,f_k), -\lambda^-(\mu,f_k)\}.
\]
If the measure $\mu$ belongs to $\mathbb{P}_{\mathrm{high}}(f_k)$, we conclude that 
\[
\min \{\lambda^+(\mu,f_k), -\lambda^-(\mu,f_k)\}>(1-\delta)\log k.
\]

 By Theorem \ref{thm.hommeasures}, we have that any measure $\nu \in \mathbb{P}_{\mathrm{high}}(f_k)$ is homoclinically related with $\mu$. In particular, the measured homoclinic class $\mathcal{H}(\mu)$ contains the set $\mathbb{P}_{\mathrm{high}}(f_k)$.

Let $\nu$ be a m.m.e. for $f_k$, by item $1$ in Theorem \ref{thm.criteriabcs} we have that $\nu= \mu$. Now let us prove that $\mu$ is Bernoulli. If it was not the case, by item $2$ of Theorem \ref{thm.criteriabcs}, there would exist a natural number $l>1$ and $l$ different probability measures $\mu_1, \cdots , \mu_l$ such that $ \mu = \frac{1}{l} (\mu_1+ \cdots + \mu_l),$ and the measures $\mu_j$ verify $(f_k)_*(\mu_j) = \mu_{j+1}$, and $(f^l_k, \mu_j)$ is Bernoulli, for $j=1, \cdots, l$. 

Since $\mu$ is a m.m.e. for $f_k$, each measure $\mu_j$ would be a m.m.e. for $f_k^l$. By Theorem \ref{thm.hommeasures}, for any $i,j \in \{1, \cdots, l\}$ the measures $\mu_i$ and $\mu_j$ are homoclinically related. Applying item $1$ of Theorem \ref{thm.criteriabcs}, we would have $\mu_i = \mu_j$. This is a contradiction, since the measures $\mu_1, \cdots, \mu_l$ were different. Hence, $l=1$ and the measure $\mu$ is Bernoulli for $f_k$.

\begin{remark}
\label{remark.homomeasures}
By Theorem \ref{thm.hommeasures}, Definition \ref{def.tophomclass} and item $2$ of Theorem \ref{thm.criteriabcs}, we actually obtain that for $k$ large enough there exists a homoclinic class $\mathrm{HC}(\mathcal{O})$ which contains the support of any measure $\nu \in \mathbb{P}_{\mathrm{high}}(f_k)$. In other words, there exists one homoclinic class that ``captures'' any measure with high enough entropy.
\end{remark}

\subsection*{$C^2$-robustness of the uniqueness of the m.m.e. for the standard map}

First, observe that the conclusions of Lemma \ref{esqueci} and Proposition \ref{sizemnfld} are $C^2$-robust, as explained in Remarks \ref{remark.robust1} and \ref{remark.robust2}. In particular, we may follow the same steps in the proof of Theorem \ref{thm.hommeasures} to obtain:

\begin{theorem}
\label{thm.robusthommeasures}
There exists $k_0>0$ such that for any $k\in [k_0, +\infty)$ there exists a $C^2$-neighborhood $\mathcal{U}$ of $f_k$, with the following property: let $g\in \mathcal{U}$, if $\mu$ and $\nu$ are two $g$-ergodic measures that verify
\[
\min \{\lambda^+(\mu), \lambda^+(\nu), -\lambda^-(\mu), -\lambda^-(\nu)\} >(1-\delta) \log k,
\]
then $\mu$ is homoclinically related to $\nu$. Furthermore, for any $l\in \N$, let $\mu_1, \cdots, \mu_j$ and $\nu_1, \cdots, \nu_s$ be the ergodic decomposition of $(g^l,\mu)$ and $(g^l,\nu)$, respectively. Then, any two measures $m_1, m_2 \in \{\mu_1, \cdots, \mu_j, \nu_1, \cdots, \nu_s\}$ are homoclinically related for $g^l$ as well. 
\end{theorem}

By Corollary \ref{cor.entropyduarte}, if $g$ is sufficiently $C^1$-close to $f_k$, we have that $h_{\mathrm{top}}(g)>(1-\delta)\log k$. Applying Theorem \ref{thm.robusthommeasures}, if $g$ is sufficiently $C^2$-close to $f_k$ then any two measures of high entropy $\mu, \nu \in \mathbb{P}_{\mathrm{high}}(g)$ are homoclinically related. In particular, if $\mu$ is a m.m.e. for $g$ then $\mathbb{P}_{\mathrm{high}}(g)$ is contained in the measured homoclinic class of $\mu$. The conclusion of Theorem \ref{thm.maintheoremuniqueness} then follows by the same arguments as for $f_k$. 

\section{Growth and equidistribution of periodic points: proof of Theorem \ref{thm.maintheoremperiodicpoints}}
\label{sec.thmb}

The goal of this section is to prove Theorem \ref{thm.maintheoremperiodicpoints}. Let us first recall a recent result by Burguet that we will use.

\begin{theorem}[Main Theorem in \cite{burguet}]
\label{thm.burguetmain}
Let $f$ be a $C^{\infty}$-diffeomorphism of a compact surface with positive topological entropy. Then for any $\rho\in (0,h_{\mathrm{top}}(f))$, it holds that:
\begin{enumerate}
\item $\displaystyle \limsup_{n\to +\infty} \frac{1}{n} \log \# \mathrm{Per}^{\rho}_n(f) = h_{\mathrm{top}}(f)$;
\item for any increasing sequence $(n_l)_{l\in \N}$ of positive integers such that 
\[
\displaystyle \lim_{l\to +\infty} \frac{1}{n_l} \log \# \mathrm{Per}^{\rho}_{n_l}(f) = h_{\mathrm{top}}(f),
\] 
any $\mathrm{weak}^*$-limit of the sequence $\left( \frac{1}{\#\mathrm{Per}^{\rho}_{n_l}(f)} \displaystyle \sum_{p\in \mathrm{Per}^{\rho}_{n_l}(f)} \delta_p \right)$ is a m.m.e. for $f$.\\
\end{enumerate}
\end{theorem}

Our goal in this section is to prove that for a $C^{\infty}$-diffeomorphism which is $C^2$-close to the standard map, we can replace in item $1$ of Theorem \ref{thm.burguetmain} by an actual limit. As we will see, the equidistribution part of Theorem \ref{thm.maintheoremperiodicpoints} will follow from that.

\begin{remark}
In Corollary $1.1$ of \cite{burguet}, Burguet obtained the following result: there exists a number $p\in \N$ such that for any $\rho \in (0,h_{\mathrm{top}}(f))$ it holds that
\[
\displaystyle \lim_{l\to +\infty} \frac{1}{lp} \log \#\mathrm{Per}^{\rho}_{lp}(f) = h_{\mathrm{top}}(f).
\]

The proof  of Corollary $1.1$ in \cite{burguet} does not imply that we may take $p$ equal $1$ in the case that there is only one m.m.e. which is Bernoulli.     
\end{remark}

We will need the following version of Katok's horseshoe theorem, which can be found in \cite{buzzicrovisiersarig}. In what follows, we will state the theorem for surface, but it is also valid for any dimension.

\begin{theorem}[Theorem $2.12$ in \cite{buzzicrovisiersarig}]
\label{thm.katok}
Let $f\in \mathrm{Diff}^r(S)$, with $r>1$, and let $\mu$ be an ergodic, non atomic, hyperbolic invariant measure. Then for any $\varepsilon>0$, there exists a basic set $\Lambda$ such that
\begin{enumerate}
\item $h_{\mathrm{top}}(f|_{\Lambda}) > h_{\mu}(f)- \varepsilon$;
\item for any $\nu \in \mathbb{P}_{\mathrm{e}}(f|_{\Lambda})$ we have that $|\lambda^*(\mu) - \lambda^*(\nu)|< \varepsilon$, for $* = -$ and $ +$;
\item if $(f,\mu)$ is mixing, then $\Lambda$ can be assumed to be topologically mixing.
\end{enumerate}

\end{theorem}

\begin{proof}[Proof of Theorem \ref{thm.maintheoremperiodicpoints}]

Let $k$, $\mathcal{U}$ be as in Theorem \ref{thm.maintheoremuniqueness}. Let $g\in \mathcal{U} \cap \mathrm{Diff}^{\infty}(\T^2)$ and let $\mu_g$ be the unique m.m.e. for $g$. Recall that from the proof of Theorem \ref{thm.maintheoremuniqueness}, we have that $h_{\mu_g}(g) > (1-\delta) \log k$, where we had fixed $\delta = \frac{1}{600}$. Fix $\rho\in (0,h_{top}(g))$. Take $\varepsilon \in (0, h_{\mu_g}(g) - (1-\delta) \log k)$ small and let $\Lambda_{\varepsilon}$ be the set given by Theorem \ref{thm.katok} applied to $\mu_g$.

By item 2 of Theorem \ref{thm.katok}, we obtain that $\mathrm{Per}_n(g|_{\Lambda_{\varepsilon}}) \subset \mathrm{Per}^{\rho}_n(g)$, for any $n\in \N$. By item $3$ of Theorem \ref{thm.katok}, we may assume that $\Lambda_{\varepsilon}$ is topologically mixing. Bowen proved  in Lemma $4$ of \cite{bowen74} that a topologically mixing basic set verifies
\[
\displaystyle \lim_{n\to +\infty} \frac{1}{n} \log \# \mathrm{Per}_n(g|_{\Lambda_{\varepsilon}}) = h_{\mathrm{top}}(g|_{\Lambda_{\varepsilon}}).
\]
Hence, by item $1$ of Theorem \ref{thm.katok} and since $\mu_g$ is the m.m.e. for $g$, we have
\[
\displaystyle \lim_{n\to +\infty} \frac{1}{n} \log \# \mathrm{Per}_n(g|_{\Lambda_{\varepsilon}}) > h_{\mathrm{top}}(g)-\varepsilon.
\]

Thus,
\begin{equation}
\label{eq.eq1}
\displaystyle \liminf_{n\to +\infty} \frac{1}{n} \log \# \mathrm{Per}^{\rho}_n(g) \geq \liminf_{n \to +\infty} \frac{1}{n} \log \# \mathrm{Per}_n(g|_{\Lambda_{\varepsilon}}) > h_{\mathrm{top}}(g) - \varepsilon.
\end{equation}
Since \eqref{eq.eq1} holds true for every $\varepsilon>0$, we conclude that 
\[
\displaystyle \liminf_{n\to +\infty} \frac{1}{n} \log \# \mathrm{Per}^{\rho}_n(g)\geq h_{\mathrm{top}}(g).
\]

By item 1 of Theorem \ref{thm.burguetmain}, we obtain
\[
\displaystyle \lim_{n\to +\infty} \frac{1}{n} \log \# \mathrm{Per}^{\rho}_n(g) = h_{\mathrm{top}}(g).
\]

Since $\mu_g$ is the unique m.m.e. for $g$, by item $2$ of Theorem \ref{thm.burguetmain}, the sequence of probability measures
\[
\displaystyle \left( \frac{1}{\# \mathrm{Per}^{\rho}_n(g)} \sum_{p\in \mathrm{Per}^{\rho}_n(g)} \delta_p \right)_{n\in \N}
\]
converges to $\mu_g$ in the $\mathrm{weak}^*$-topology.
\end{proof}

\section{Further properties of the m.m.e.: proof of Theorems \ref{thm.maintheoremdensity} and \ref{thm.genericparamenters}} \label{sec.thmcd}

In this section we prove Theorems \ref{thm.maintheoremdensity} and \ref{thm.genericparamenters}. For this, we will use several of the properties stated in Section \ref{sec.standardback}.

The relation between Lyapunov exponents, the entropy and the Hausdorff dimension of a hyperbolic measure have been well studied (see for instance \cite{young82, barreirapesins}). We will need the following formula of the dimension of the measure for surfaces.

\begin{theorem}[Main theorem in \cite{young82}]
\label{thm.young}
Let $f:S\to S$ be a $C^2$-diffeomorphism of a compact surface $S$. Suppose that $\mu$ is an ergodic measure such that $h_{\mu}(f)>0$. Then
\[
\mathrm{dim}_H(\mu) = h_{\mu}(f)\left(\frac{1}{\lambda^+(\mu)} - \frac{1}{\lambda^-(\mu)}\right).
\]
\end{theorem}

\begin{proof}[Proof of Theorem \ref{thm.maintheoremdensity}]
Fix $\varepsilon>0$. First, if $\delta>0$ is small  and $k \in \R$ is large enough then
\begin{equation}
\label{eq.estimate}
\displaystyle \frac{(1-\delta)}{1 + \frac{\log 4\pi}{\log k}} > 1-\frac{\varepsilon}{2}.
\end{equation}

Fix some small $\delta>0$. Let $k_0\in \R$ be a number given by Corollary \ref{cor.entropyduarte} applied to $\delta$. We may also suppose that $k_0$ is large enough such that $k_0$ verifies Theorem \ref{thm.maintheoremuniqueness} and for any $k\in [k_0, +\infty)$ the inequality \eqref{eq.estimate} holds. Fix $k\in [k_0, + \infty)$ and for each diffeomorphism $g$ sufficiently $C^1$-close to $f_k$, let $\Lambda_k(g)$ be the basic set given by Corollary \ref{cor.entropyduarte}. 

By Theorem \ref{thm.maintheoremuniqueness}, if $g\in \mathrm{Diff}^2(\T^2)$ is sufficiently $C^2$-close to $f_k$, then $g$ has at most one m.m.e. Suppose that $\mu$ is an m.m.e. for $g$ and observe that $h_{\mu}(g)> (1-\delta) \log k$. By \eqref{eq.c1norm}, we also have that $h_{\mu}(g) \leq \log 4\pi k.$ By Theorem \ref{thm.young}, we obtain
\[\arraycolsep=1.2pt\def\arraystretch{2.5}
\begin{array}{rcl}
\mathrm{dim}_H(\mu) &= & \displaystyle h_{\mu}(g) \left(\frac{1}{\lambda^+(\mu)} - \frac{1}{\lambda^-(\mu)} \right)> (1-\delta)\log k \left(\frac{2}{\log 4\pi k}\right)\\ 
& = & \displaystyle 2\left( \frac{1-\delta}{1 + \frac{ \log 4\pi }{\log k}} \right) > 2\left(1-\frac{\varepsilon}{2}\right)  = 2-\varepsilon.
\end{array}
\]
This completes the proof of item $1$ from Theorem \ref{thm.maintheoremdensity}. 

Since $\Lambda_k(g)$ is a basic set, it is well known that the periodic points are dense in it, that is, $\overline{ \mathrm{Per}_h(g|_{\Lambda_k(g)})} = \Lambda_k(g)$ (see for instance \cite{newhouse2}). Since $h_{\mathrm{top}}(g|_{\Lambda_k(g)})> (1-\delta) \log k$, take $\nu$ to be an ergodic measure supported in $\Lambda_k(g)$ such that $h_{\nu}(g)> (1-\delta) \log k$. By Theorem \ref{thm.robusthommeasures}, we have that the measure $\nu$ is homoclinically related with $\mu$.

By Theorem \ref{thm.criteriabcs}, there exists a hyperbolic periodic orbit $\mathcal{O}$ which is homoclinically related to $\mu$ and such that $\mathrm{supp}(\mu) = \mathrm{HC}(\mathcal{O})$. This implies that $\nu$ is also homoclinically related with $\mathcal{O}$. Since $\Lambda_k(g)$ is a basic set, it is easy to conclude that every periodic orbit in $\mathrm{Per}_h(g|_{\Lambda_k(g)})$ is homoclinically related to $\mathcal{O}$. Hence, $\Lambda_k(g) \subset \mathrm{HC}(\mathcal{O})$. By Corollary \ref{cor.entropyduarte}, the set $\Lambda_k(g)$ is $\frac{8}{k^{\frac{1}{3}}}$-dense in $\T^2$, and this concludes the proof of Theorem \ref{thm.maintheoremdensity}.

\end{proof}

\begin{proof}[Proof of Theorem \ref{thm.genericparamenters}]

Let $k_0$ be large enough such that Theorems \ref{thm.maintheoremuniqueness} and \ref{thm.gorod} holds. Let $\mathcal{R} \subset [k_0, + \infty)$ be the set of parameters given by Theorem \ref{thm.gorod}, and fix $k\in \mathcal{R}$. 

Let $\Lambda_k^{(0)} \subset \Lambda_k^{(1)} \subset \cdots $ be the sequence of basic sets and $\Omega_k$ be given by the conclusion of Theorem \ref{thm.gorod} for $f_k$. Let also $\mu$ be the unique m.m.e. for $f_k$, given by Theorem \ref{thm.maintheoremuniqueness}.

Arguing similarly as in the proof of Theorem \ref{thm.maintheoremdensity}, there exists a hyperbolic periodic orbit $\mathcal{O}$ which is homoclinically related to $\mu$ and such that for any $n\in \N$ we have
\[
\Lambda_k^{(n)} \subset \mathrm{HC}(\mathcal{O}) = \mathrm{supp}(\mu).
\]
In particular, since $\Omega_k = \overline{ \bigcup_{n\in \N} \Lambda_k^{(n)}}$, we obtain that $\Omega_k \subset \mathrm{supp}(\mu)$. By item $2$ of Theorem \ref{thm.gorod}, we conclude that 
\[
\mathrm{dim}_H(\mathrm{supp}(\mu)) \geq \mathrm{dim}_H(\Omega_k) = 2.
\]
Hence $\mathrm{dim}_H(\mathrm{supp}(\mu)) = 2$.
\end{proof}

\begin{remark}
Observe that in Theorem \ref{thm.genericparamenters}, we obtain that the set $\Omega_k$ is contained in the support of the unique m.m.e. for $f_k$. An interesting question would be to know if the support actually coincides with the set $\Omega_k$. In particular, this would imply that for a generic large parameter any point in the support of the unique m.m.e. is accumulated by elliptic islands.
\end{remark}

\address
\end{document}